\documentclass{article}
\usepackage{amsmath,amssymb,amsthm,scrtime,bbm,latexsym,verbatim,mathrsfs}
\usepackage{graphics}
\usepackage{amsfonts}
\usepackage[usenames,dvipsnames]{xcolor}
\usepackage{tikz}
\usepackage{bm}
\usepackage{csquotes}
\MakeOuterQuote{"}
\PassOptionsToPackage{bookmarksopen,colorlinks,linkcolor=black}{hyperref}
\usepackage[textsize=small]{todonotes}
\usepackage[bookmarksopen,colorlinks,linkcolor=black]{hyperref}

\newcommand{\E}{\mathbb{E}}
\renewcommand{\P}{\mathbb{P}}
\newcommand{\N}{\mathbb{N}}
\newcommand{\Z}{\mathcal{Z}}
\newcommand{\R}{\mathbb{R}}
\newcommand{\eps}{\epsilon}
\newcommand{\sss}{\scriptscriptstyle}
\newcommand{\F}{\mathcal{F}}

\newcommand{\supnorm}[1]{\| #1\|}

\newcommand{\graph}{{\sf{G}}}
\newcommand{\connect}{{\sf{C}}}

\newcommand{\pc}{p_{\rm c}}
\newcommand{\tree}{\Gamma}

\newcommand{\eigenfct}{v} 
\newcommand{\IBRW}{{\mathrm{IBRW}}} 
\newcommand{\size}{\theta}
\newcommand{\interval}{\mathcal{I}}
\newcommand{\betac}{\beta_{\rm c}}
\newcommand{\gammac}{\gamma_{\rm c}}
\newcommand{\mUnbiased}{\mathtt{a}}
\newcommand{\mBiased}{\mathtt{c}}
\newcommand{\stepconst}{a}
\newcommand{\bdconst}{b}

\newcommand{\asymptotically}{{\rm asym}}

\newcommand{\OffBase}{\mathcal{C}_{n}} 
\newcommand{\offDistrBase}{X_{n}} 
\newcommand{\subGW}{{\rm GW}_{n}} 
\newcommand{\offDistr}{\offDistrBase}

\newcommand{\subOff}{\OffBase}
\newcommand{\OffAsymp}{\OffBase^{\sss (\asymptotically)}}

\newcommand{\extinctProbGW}{q_{n,\eps}}
\newcommand{\momGen}{g}

\newcommand{\parent}{{\rm p}}
\newcommand{\loc}{S} 
\newcommand{\type}{\tau} 
\newcommand{\typespace}{\mathcal{T}}

\newcommand{\set}[1]{\left\{#1\right\}}

\newcommand{\abs}[1]{\left\vert#1\right\vert}

\newcommand{\la}{\lambda}
\newcommand{\ra}{\rightarrow}
\newcommand{\p}{\mathbb{P}}
\newcommand{\cC}{\mathcal{C}}
\newcommand{\1}{\mathbbm{1}}
\newcommand{\ssup}[1] {{\scriptscriptstyle{({#1})}}}

\numberwithin{equation}{section}

\newtheorem{theorem}{Theorem}[section]
\newtheorem{lemma}[theorem]{Lemma}

\newtheorem{corollary}[theorem]{Corollary}
\newtheorem{proposition}[theorem]{Proposition}

\theoremstyle{definition}

\newtheorem{remark}[theorem]{Remark}

\definecolor{MyDarkBlue}{rgb}{0,0.08,0.50}
\definecolor{BrickRed}{rgb}{0.65,0.08,0}
\hypersetup{
colorlinks=true,       
    linkcolor=MyDarkBlue,          
    citecolor=BrickRed,        
    filecolor=red,      
    urlcolor=cyan           
}

\title{Near critical preferential attachment networks have\\ small giant components}
\author{Maren Eckhoff\thanks{QuantumBlack,  Kinnaird House, 1 Pall Mall East, London SW1Y 5AU, United Kingdom.\newline Email: {\tt eckhoff.maren@gmail.com}}
\and
Peter M\"orters\thanks{Mathematisches Institut, Universit\"at zu K\"oln, Weyertal 86-90, 50931 K\"oln, Germany.\newline Email: {\tt moerters@math.uni-koeln.de}}
\and
Marcel Ortgiese\thanks{Department of Mathematical Sciences, University of Bath, Bath, BA2 7AY, United Kingdom.
\newline Email: {\tt m.ortgiese@bath.ac.uk} }}

\begin{document}
\maketitle

\begin{abstract}
\noindent
Preferential attachment networks with power law exponent~$\tau>3$ are known to exhibit a phase transition. There is a value $\rho_{\rm c}>0$ such that, for small edge densities $\rho\leq \rho_{\rm c}$ every component of the graph comprises an asymptotically vanishing proportion of vertices, while for large edge densities $\rho>\rho_{\rm c}$  there is a unique giant component comprising an asymptotically  positive proportion of vertices. In this paper we study the decay in the size of the giant component as the critical edge density is approached from above. We show that the size decays very rapidly, like $\exp(-c/ \sqrt{\rho-\rho_{\rm c}})$ for an explicit constant~$c>0$ depending on the model implementation. This result is in contrast to 
the behaviour of the class of rank-one models of scale-free networks, including the configuration model, where the decay is polynomial. Our proofs rely on the local neighbourhood approximations of~\cite{DerMoe13} and recent progress in the theory of branching random walks~\cite{GHS11}.
\end{abstract}

\bigskip



\section{Introduction and main results}

\subsection{Introduction}

Sparse random graph models typically undergo a phase transition in their connectivity behaviour depending on the mean number of edges per vertex. A typical case are the Erd\H{o}s-Renyi graphs with $n$ vertices and $m=m(n)$ edges and asymptotic edge density \smash{$\rho=\lim \frac{m(n)}{n}$}. There exists a critical density $\rho_{\rm c}=1$ such that if the edge density satisfies $\rho\leq \rho_{\rm c}$ the largest component in the graph comprises a vanishing proportion of vertices, whereas for $\rho>\rho_{\rm c}$ this proportion~$\theta(\rho)$ is strictly positive. The behaviour of $\theta(\rho)$ as $\rho\downarrow \rho_{\rm c}$ is characterised by an exponent~$\beta$ defined by
$$\theta(\rho)\sim (\rho-\rho_{\rm c})^\beta, \qquad \mbox{ as } \rho\downarrow \rho_{\rm c},$$
which is $\beta=1$ in the Erd\H{os}-Renyi case. A natural extension of the Erd\H{os}-Renyi model is the configuration model, which allows to construct random graphs with $n$ vertices and a given degree 
sequence $m_1,\ldots, m_n$. Of particular interest is the case when the degree sequence is heavy-tailed,
$$\frac1n \# \{ 1\leq j\leq n \colon m_j =k\} \stackrel{n\to\infty}{\longrightarrow} \mu(k)=k^{-\tau+o(1)}, \quad \mbox{ as } k\to\infty,$$
where the parameter $\tau>2$ is the power-law exponent. The connectivity behaviour of the
Erd\H{o}s-Renyi model persists with $\beta=1$ for the configuration model with $\tau>4$, and with
$\beta=1/(\tau-3)$ if $3<\tau<4$, see Cohen et al.~\cite{CbAH02}. The paper by Cohen is based on
an informal approximation of local neighbourhoods in the graph by Galton-Watson trees and thereby
extends to  a wide range of scale-free network models, where similar approximations hold.
\smallskip

Cohen et al.~\cite{CbAH02} claim their result for scale-free networks in general without specifying a model.
This reflects the belief that the behaviour observed in the configuration model extends to all natural scale-free network models including the class of preferential attachment networks.   In the present paper however we show for the first time that
preferential attachment networks have a qualitatively completely different behaviour than predicted in~\cite{CbAH02}. In fact, for all $\tau>3$ the size of the giant component is decaying exponentially as one approaches the critical edge density. More precisely, we show that the relative size of the giant component in preferential attachment networks~is
$$\theta(\rho) = \exp\Big(\frac{-c+o(1)}{\sqrt{\rho-\rho_{\rm c}}}\Big), \mbox{ as $\rho\downarrow\rho_{\rm c},$}$$
where $c$ is an explicit constant depending on the way in which the edge density is controlled.
This demonstrates once again  that preferential attachment networks belong to a different universality class than the configuration model and other models based on rank-one connection probabilities.
\smallskip

The underlying phenomenon of the `small giant component' or `slow emergence of the giant' has first been discovered and discussed by Oliver Riordan in the seminal paper~\cite{Rio05}, and in collaboration with Bollob\'as and Janson in~\cite{BolJanRio05} and \cite{BR05}.  Riordan~\cite{Rio05} finds that for the original Barabasi-Albert model subject to Bernoulli percolation with retention parameter~$p$, one has
\begin{equation}\label{LCDresult}
\size(p)= \exp\Big(-\frac{c_m +o(1)}{p}\Big) \quad \text{as }p \downarrow 0, \quad \text{ for }c_m=\frac{1}{2}\log\Big(\frac{m+1}{m-1}\Big), 
\end{equation}
where $m\geq 2$ is the outdegree of all vertices. As this model corresponds to the critical case $\tau=3$ this is not at odds with the results of Cohen et al.~\cite{CbAH02}. The merit of our work is to extend the phenomenon of slow emergence to a regime where  it is most surprising because it defies the predictions in~\cite{CbAH02}.
\smallskip

In his proof, Riordan exploits the fact that there are local approximations of the network by multitype branching processes whose survival probability can be studied analytically by looking at the associated Laplace operators. These branching processes are more complex than those used in~\cite{CbAH02}. 
\smallskip

In the present paper local approximation is by a branching random walk with killing. Instead of the analytical techniques used by Riordan, which we could not apply effectively in our case due to the higher complexity of the approximating process, we use geometric properties of killed branching random walks. Studying their survival probabilities requires sophisticated techniques from the theory of branching random walks, which became available only in the past few years, see in particular~\cite{GHS11}. Our results are therefore pleasingly probabilistic and highly timely.

\subsection{Statement of results}

We start by describing the preferential attachment network introduced in~\cite{DerMoe09}, which gives scale-free networks with arbitrary power law exponent~$\tau>2$ by variation of a parameter.
\smallskip

A concave function $f\colon \N_0 \to (0,\infty)$ is called an \emph{attachment rule} if $f(0)\le 1$ and $$\triangle f(k):=f(k+1)-f(k) <1 \mbox{  for all $k\in \N_0$. }$$
The maximal increment is denoted by $\gamma^+:=\sup\{\triangle f(k)\colon k\in \N_0\}$. By concavity, $f$ is non-decreasing, $\gamma^+=f(1)-f(0)$ and the limit $\gamma:=\lim_{k \to \infty} {f(k)}/{k}$ exists and equals $\gamma=\inf\{\triangle f(k)\colon k\in \N_0\}$.
Given an attachment rule $f$,
we define a growing sequence $(\graph_n \colon n\in \N)$ of random graphs as follows
\begin{itemize}
\item  Start with the graph $\graph_1$ given by one vertex labelled $1$ and no edges;
\item Given the graph $\graph_n$, we construct $\graph_{n+1}$ from $\graph_n$ by adding a new vertex labelled $n+1$ and, for each $m \le n$ independently, 
inserting the directed edge $(n+1,m)$ with probability
\[
\frac{f(\text{indegree of } m \text{ at time }n)}{n}.
\] 
\end{itemize}

Formally, we are dealing with a sequence of directed graphs but all edges point from the younger to the older vertex. 
Hence, directions can be recreated from the undirected, labelled graph. For all structural questions, particularly regarding 
connectivity and the length of shortest paths, we regard $(\graph_n \colon n\in\N)$ as an undirected network. It is shown in~\cite{DerMoe09} that when $\gamma>0$ the networks have a degree distribution which is a power law with exponent \smash{$\tau=1+1/\gamma.$}
We are also interested in the percolated version of the network $(\graph_n\colon n \in \N)$. For $p \in [0,1]$, we write $\graph_n(p)$ for the graph obtained from $\graph_n$ by deleting each edge with probability $1-p$ independent of all other edges.
\smallskip

Let $(\graph_n\colon n \in \N)$ be a sequence of (random or deterministic) graphs, where $\graph_n$ has $n$ vertices. For $n \in \N$, we denote by $|\connect_n|$ the size of the largest component in $\graph_n$. The network $(\graph_n\colon n \in \N)$ has a \emph{giant component} if there exists a constant $\size >0$ such that
\[
\frac{|\connect_n|}{n} \to \size \qquad \text{as }n \to \infty,
\]
where the convergence holds in probability. The limit $\theta$ is called the \emph{size of the giant component}. 
\smallskip

Dereich and M\"orters showed in \cite[Theorem 1.6]{DerMoe13} that when $\gamma\ge \frac{1}{2}$, or equivalently $2<\tau \leq 3$, then $(\graph_n(p) \colon n \in \N)$ has a giant component for all $p \in (0,1]$. 
When $\gamma <\frac{1}{2}$, or equivalently $\tau>3$, then there exists a critical percolation parameter $\pc>0$ such that $(\graph_n(p) \colon n \in \N)$ has giant component if and only if $p>\pc$. We denote the relative size of the giant component in $(\graph_n(p) \colon n \in \N)$ by $\size(p,f)$ and omit $p$ or $f$ from the notation when the percolation parameter or the attachment rule are fixed.
\smallskip

We are interested in the decay in the size of the giant component
as we approach $\pc$ from above in the case $\pc>0$, or equivalently $\gamma <\frac{1}{2}$.
It was shown in Lemma 3.3 of \cite{DerMoe13} that the critical retention probability for the network $(\graph_n \colon n \in \N)$ is given by 
\[
\pc=\rho(\alpha^*)^{-1},
\]
where $\rho(\cdot)$ is the spectral radius of the score operator, a function on a nonempty  open interval, which we describe explicitly in Section~2, and $\alpha^*$ is the minimizer of $\rho$.
\smallskip

Our first result shows the exponential decay of the size of the giant component of the percolated network, when the retention parameter approaches the critical value.

\begin{theorem}\label{thm:percol}
Let $f$ be an attachment rule with $\gamma<\frac{1}{2}$ and $\size(1,f)>0$. Then
\[
\lim_{p \downarrow \pc} \sqrt{p-\pc} \log \size(p,f)  = -\sqrt{\frac{\pi^2 \rho''(\alpha^*)}{2}} \alpha^* \pc.
\]
\end{theorem}

An alternative way of reducing the edge density and thereby
destroying the giant component is to alter the attachment rule instead of percolating the network. 
For a linear attachment rule $f(k)=\gamma k+\beta$ with $\gamma <\frac{1}{2}$, Dereich and M\"orters \cite{DerMoe13} show that there exists a giant component if and only if
\[
\beta > \frac{(\frac{1}{2}-\gamma)^2}{1-\gamma}=:\betac(\gamma)=\betac.
\]
Therefore, one could fix $\gamma$ and decrease $\beta$ to $\betac$. Another idea would be, for a given $\beta$, to decrease $\gamma$ until $\beta=\betac(\gamma)$. 
To analyse the behaviour of the size of the giant component under this procedure, let $(f_t)_{t \ge 0}$ be a sequence of attachment rules with $\gamma_t:=\inf\{\triangle f_t(k)\colon k \in \N_0\}<\frac{1}{2}$ for all $t \ge 0$. We denote by $\rho_t$ and $\alpha_t^*$ the spectral radius and its minimizer corresponding to the score operator for the unpercolated branching random walk derived from attachment rule $f_t$.

\begin{theorem}\label{thm:vary_f}
Let $(f_t)_{t\ge 0}$ be a pointwise decreasing sequence of attachment rules with $\gamma_t<\frac{1}{2}$ for all $t\ge 0$ and pointwise limit $f$. Suppose that $\size(1,f_t)>0$ for all $t$ and $\size(1,f)=0$. Then
\[
\lim_{t \to \infty} \sqrt{\log \rho_t(\alpha_t^{*})} \log \size(1,f_t) =  - \sqrt{\frac{\pi^2 \rho''(\alpha^*) }{2}}\alpha^*,
\]
where $\alpha^*$ and $\rho''(\alpha^*)$ are derived from $f$.
\end{theorem}

The existence of $\rho$ and $\alpha^*$ corresponding to $f$ is proved in  Proposition~\ref{prop:conv}.
 There we will also see that $\lim_{t \to \infty}\rho_t(\alpha_t^*)=1$.
The following corollary exemplifies Theorem~\ref{thm:vary_f} for linear attachment rules. We denote
\[
\betac(\gamma)=\frac{(\frac{1}{2}-\gamma)^2}{1-\gamma}\quad \text{and} \quad \gammac(\beta)=\frac{1}{2} \big(1-\beta -\sqrt{\beta^2+2\beta}\big).
\]

\begin{corollary}[Linear attachment rule]\label{linear_cor}
Let $\gamma \in [0,\frac{1}{2})$, $\beta \in (0,1]$. Then
\begin{equation}
\lim_{\beta \downarrow \beta_c(\gamma)} \sqrt{\beta-\beta_c(\gamma)} \log \size(1,\gamma \cdot+\beta) = - \frac{\pi}{2\sqrt{1-\gamma}}.\label{beta_cor}
\end{equation}
If $\beta \in (0,1/4]$, then
\begin{equation}
\lim_{\gamma \downarrow \gamma_c(\beta)} \sqrt{\gamma-\gamma_c(\beta)} \log \size(1,\gamma \cdot+\beta) = - \frac{\pi}{2 (\beta^2+2\beta)^{1/4}}.\label{gamma_cor}
\end{equation}
\end{corollary}

\begin{figure}[h]
\centering
\includegraphics[trim = 0.5cm 0.5cm 0.5cm 0.5cm, width=9cm]{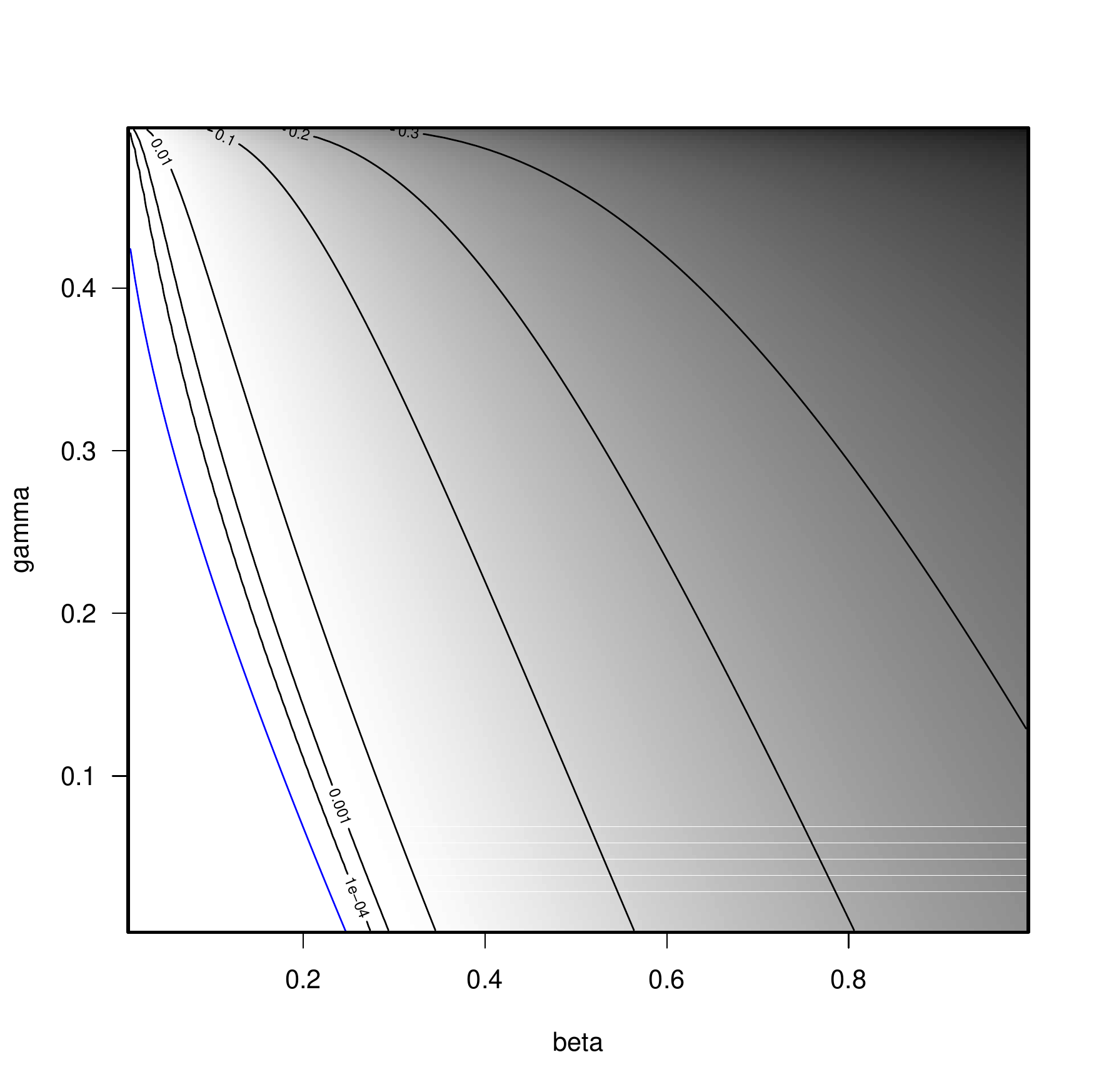}
\caption{Theorem~\ref{thm:vary_f} allows us to approach the critical line from above, from the right and from any angle in between. }
\end{figure} 

\newpage
\begin{remark}
Two cases in our phase diagram are covered in the work of Riordan~\cite{Rio05}. The first of these cases corresponds to an approximation from the right of the point $\beta=0, \gamma=\frac12$ which is equivalent to the original Barabasi-Albert, or LCD, model.\footnote{There is a difference in the set-up of the LCD model and our model, as the former uses a fixed number of connections for every new vertex. This technical  difference does not affect the results beyond the form of the constants involved.}  Note that our results refer to the subcritical case $\gamma<\frac12$ and the critical case $\gamma=\frac12$ is not included. The second is the case $\beta=\frac14, \gamma=0$, the Dubins model, in which there is no preferential attachment  and our results are consistent with those of Riordan~\cite{Rio05}. 
\smallskip

Corollary~\ref{linear_cor} allows a quantitative comparison of the decay of the giant component for different models. The smaller $\gamma$  (or the larger~$\tau$), the slower is the decay. The LCD model, or equivalent models with $\gamma=\frac{1}{2}$, have faster decay of the size of the giant component than  
preferential attachment networks with attachment rules satisfying $\gamma <\frac{1}{2}$.
\end{remark}

Throughout, we will use the following notation: For every integer $k\in \N$, we write $[k]=\{1,\dots,k\}$, $\R_+:=[0,\infty)$, $\N_0:=\N \cup \{0\}$.

The remaining paper is structured as follows: We prove Theorems~\ref{thm:percol} and \ref{thm:vary_f} simultaneously. In Section~\ref{sec:preparations} we collect several auxiliary results that we need later on. 
In particular, in Section~\ref{sec:BP_def} we recall the relevant results from~\cite{DerMoe13} who relate the size of the giant component to the survival probability of a 
multitype branching random walk with killing. 
In Section~\ref{sec:mto} we derive the main tool to analyse these branching random walks, a version of the well-known many-to-one lemma.
In Section~\ref{ssn:asymp_moments} we collect various moment estimates, while in 
Section~\ref{sec:mogulskii} we state a large deviation result originally due to Mogulskii in a suitable adaptation to our setting. 
The actual proofs of Theorems~\ref{thm:percol} and~\ref{thm:vary_f} are
split into an upper bound carried out in Section~\ref{sec:upper_bound} and 
a lower bound in Section~\ref{sec:lower_bound}.
In Section~\ref{sec:linear} we show how to derive Corollary~\ref{linear_cor}
from Theorem~\ref{thm:vary_f}.
Finally, in Appendix~\ref{mogulskii_proof} we prove the large deviation result, Theorem~\ref{thm:mogulskii}.

\section{Proofs: General Preparations}\label{sec:preparations}

\subsection{The approximating branching process}\label{sec:BP_def}

We introduce a pure jump Markov process 
with generator
\[
Lg(k):=f(k)\big( g(k+1)-g(k)\big).
\]
This defines an increasing, integer-valued process, which jumps from $k$ to $k+1$ after an exponential waiting time with mean $1/f(k)$, independently of the previous jumps. 
Under the probability $P$ we denote by $(Z_t\colon t \ge 0)$ the process started in zero,  by $(\hat{Z}_t\colon t\ge 0)$ the process started in one, and by $(Z_t^{\ssup \tau} \colon t \ge 0)$
the process started in zero conditioned to have a point at $\tau\geq0$.\smallskip

This process is used to define a multitype branching random walk with type space $\typespace:=[0,\infty)\cup\{\ell\}$, 
where $\ell$ is a non-numerical symbol for `left'. 
A particle in location $x \in \R$ and of type $\type \in \typespace$, produces offspring to its left whose displacements have the same distribution as the points of the Poisson point process with intensity measure
\[
e^t E[f(Z_{-t})] \mathbbm{1}_{(-\infty,0]}(t) \, dt.
\] 
The type of an offspring on the left equals the distance to its parent. 
\smallskip

The distribution of the offspring to the right depends on the type of the particle. When the particle is of type $\ell$, then the relative positions of its right offspring follow the same distribution as the jump times of $(Z_t\colon t\ge 0)$. When the particle is of type $\type \ge 0$, then the displacements follow the same distribution as the jump times of $(Z_t^{\sss (\type)}-\mathbbm{1}_{[\type,\infty)}(t)\colon t\ge 0)$. All offspring on the right are of type $\ell$.
\smallskip

The offspring to the right do \emph{not} form a Poisson point process. The more particles are born, the higher the rate of new particles arriving. Moreover, the total number of particles produced is infinite without accumulation point. 
The expected distance between a particle and its $k$th offspring on the right equals $\sum_{j=0}^{k-1}\frac{1}{f(j)}= \sum_{j=0}^{k-1} \frac{1}{f(j)/j} \frac{1}{j}$. Since $\lim_{j \to \infty} \frac{f(j)}{j} = \gamma$, this distance behaves asymptotically like $\gamma^{-1} \log(k)$ when $\gamma \not=0$ and like $k$ when $\gamma =0$.
\smallskip 

We call the described process \emph{idealized branching random walk} ($\IBRW$) in accordance with \cite{DerMoe13}. Dereich and M\"orters \cite{DerMoe13} show that the genealogical tree of the $\IBRW$ is related to the local neighbourhood of a vertex in $\graph_n$. 
To obtain a branching process approximation to $(\graph_n(p)\colon n \in \N)$, we define the \emph{percolated} $\IBRW$ by associating to every offspring in the $\IBRW$ an independent Bernoulli$(p)$ random variable. 
If the random variable is zero, we delete the offspring together with its descendants. Otherwise, the offspring is retained in the percolated $\IBRW$. 
When the percolated $\IBRW$ is started with one particle in location $x$ and type $\type$, then we write $P_{(x,\type)}^p$ for its distribution and $E_{(x,\type)}^p$ for the corresponding integral operator; 
\smash{$P_{(x,\type)}:=P_{(x,\type)}^1$, $E_{(x,\type)}:=E_{(x,\type)}^1$.}
\smallskip

The percolated $\IBRW$ can be interpreted as a labelled tree $\tree$ where every node represents a particle and is connected to its children and (apart from the root) to its parent. 
The vertices are identified as finite sequences of natural numbers $x=j_1\ldots j_k$, including the empty sequence $\emptyset$ which denotes the root. 
We concatenate sequences $x=i_1\ldots i_k$ and $y=j_1\ldots j_m$ to form the sequence $xy=i_1\ldots i_kj_1\ldots j_m$. 
When, for $j \in \N$, $xj$ is a vertex in $\tree$, then $x,x1,\ldots,x(j-1)$ are also vertices in the tree and $x=:\parent(xj)$ is the parent of $xj$. 
The length $|x|=k$ is the generation of $x$. 
For $|x|\ge k$, we abbreviate $\parent^k(x)$ for the $k$-fold composition of $\parent(\cdot)$. 
The ancestor of $x$ in generation $k$ is denoted by $x_k$, i.e.\ $x_k:=\parent^{n-k}(x)$ when $|x|=n$. 
In particular, $x_0$ always denotes the root. 
To every vertex we associate two functions, $\loc$ and $\type$. Here $\loc(x)$ is the location of the particle on the real line and $\type(x)$ denotes its type.
\smallskip

To obtain a branching process approximation to the local neighbourhood of a vertex in $(\graph_n(p)\colon n \in \N)$, we consider the percolated $\IBRW$ with a killing barrier at zero. That is, every particle with location on the nonnegative half-line is deleted together with its descendants. Dereich and M\"orters prove the following identification.

\begin{theorem}[{\bf Dereich and M\"orters \cite{DerMoe13}}]
For all $p\in [0,1]$ and attachment rules $f$, $\size(p,f)$ equals the survival probability of the percolated $\IBRW$ with a killing barrier at zero, started with one particle of type $\ell$ whose location is given by $-E$, where $E$ is an exponential random variable with mean one.
\end{theorem}

Next we collect some spectral properties that will be used in the analysis of the IBRW.

Denote by $C(\typespace)$ the Banach space of bounded, continuous functions on $\typespace$ equipped with the supremum norm $\supnorm{\cdot}$. For $\alpha \in (0,1)$, we consider the score operator
\begin{equation}\label{def:operator}
A_{\alpha}^pg(\type_0) =E_{(0,\type_0)}^p\Big[\sum_{|x|=1} g(\type(x)) e^{-\alpha \loc(x)} \Big], \qquad \type_0 \in \typespace,
\end{equation}
on $C(\typespace)$, $A_{\alpha}:=A_{\alpha}^1$. The spectral radius of $A_{\alpha}^p$ is denoted by $\rho^p(\alpha)$. The dependence of $A_{\alpha}^p$ on the attachment rule $f$ is suppressed in notation but it will always be clear from the context which $f$ is considered. Since by definition,
\begin{equation}\label{eq:ApvsA}
A_{\alpha}^p g(\type_0)=E_{(0,\type_0)}^p\Big[\sum_{|x|=1} g(\type(x)) e^{-\alpha \loc(x)} \Big]= p E_{(0,\type_0)}\Big[\sum_{|x|=1} g(\type(x)) e^{-\alpha \loc(x)} \Big]=p A_{\alpha}g(\tau_0),
\end{equation}
it suffices to analyse $A_{\alpha}$. We write $\mathbf{1}$ for the constant function with value $1$ and let $\interval:=(\gamma,1-\gamma)$ for $\gamma <\frac{1}{2}$ and $\interval=\emptyset$ for $\gamma \ge \frac{1}{2}$.

\begin{lemma}\label{lem:SpecProp} 
Let $p \in (0,1]$.
If $\gamma \ge \frac{1}{2}$, then $A_{\alpha}^p\mathbf{1}(0)=\infty$. If $\gamma <\frac{1}{2}$, then the following holds:
\begin{itemize}
\item[{\rm(i)}] $\rho^p(\alpha)$ is finite for all $\alpha \in \interval$, $\rho^p(\alpha)=p\rho(\alpha)$ and $\rho(\alpha)\to \infty$ for $\alpha \to \partial \interval$.
\item[{\rm(ii)}]There exists a unique positive eigenfunction $\eigenfct_{\alpha}$ of $A_{\alpha}^p$ corresponding to $\rho^p(\alpha)$ with $\supnorm{\eigenfct_{\alpha}}=1$. Moreover, $\eigenfct_{\alpha}$ does not depend on the retention probability $p$ and $\min_{\tau \in \typespace} v_\alpha(\tau) > 0$.
\item[{\rm(iii)}] The function $\rho$ is twice differentiable on $\interval$ with 
\begin{equation}\label{diff_eq}
\rho^{(i)}(\alpha)=E_{(0,\type_0)}\Big[\sum_{|x|=1} \frac{\eigenfct_{\alpha}(\type(x))}{\eigenfct_{\alpha}(\type_0)} e^{-\alpha \loc(x)} (-\loc(x))^{i}\Big] \qquad \text{for all } \type_0 \in \typespace, i \in \{1,2\}.
\end{equation}
\item[{\rm (iv)}] The function $\rho$ is strictly convex on $\interval$ and there exists a unique minimizer $\alpha^* \in \interval$.
\item[{\rm (v)}] For any $\tau \in [0,\infty)$ 
\[ v_\alpha(\ell) \leq  v_\alpha(\tau) \leq v_\alpha(0)  . \]
\end{itemize}
\end{lemma}

\begin{proof}
By \eqref{eq:ApvsA}, it suffices to prove Lemma~\ref{lem:SpecProp} in the case $p=1$. For that case, it was shown in \cite[Lemma 3.1]{DerMoe13} that $A_{\alpha}1(0)<\infty$ is equivalent to $A_{\alpha}$ being a strongly positive, compact operator with $A_{\alpha}g \in C(\typespace)$ for all $g \in C(\typespace)$. Moreover, it is proved that for all $g \in C(\typespace)$, $g\ge 0$, we have $\underline{A}_{\alpha}g \le A_{\alpha}g \le \overline{A}_{\alpha}g$, where
\[
\underline{A}_{\alpha}=\begin{pmatrix}
\mUnbiased(\alpha) & \mUnbiased(1-\alpha)\\
\mUnbiased(\alpha) & \mUnbiased(1-\alpha)\end{pmatrix}, \qquad 
\overline{A}_{\alpha}=\begin{pmatrix}
\mUnbiased(\alpha) & \mUnbiased(1-\alpha)\\
\mBiased(\alpha) & \mUnbiased(1-\alpha) \end{pmatrix},
\]
and
\[
\mUnbiased(\alpha)=\int_0^{\infty} e^{-\alpha t} E[f(Z_t)] \, dt, \qquad \mBiased(\alpha)=\int_0^{\infty} e^{-\alpha t} E[f(\hat{Z}_t)]\, dt.
\]
Here, for example, $\overline{A}_{\alpha}g$ means that for all $\type \ge 0$,
\[
\overline{A}_{\alpha}g(\type)=\int_0^{\infty} g(\ell) e^{-\alpha t} E[f(\hat{Z}_t)]\, dt+\int_0^{\infty} g(t) e^{-(1-\alpha)t}E[f(Z_t)]\, dt.
\]
We have, $\rho(\alpha)\ge \rho(\underline{A}_{\alpha})=\mUnbiased(\alpha)+\mUnbiased(1-\alpha)$.
The values of $\mUnbiased$ and $\mBiased$ are identified as (cf.\ proof of Proposition 1.10 in \cite{DerMoe13})
\begin{equation}\label{ac_fctns}
\mUnbiased(\alpha)=\sum_{k=0}^{\infty} \prod_{j=0}^k \frac{f(j)}{f(j)+\alpha}, \qquad \mBiased(\alpha)=\sum_{k=0}^{\infty} \prod_{j=0}^{k} \frac{f(j+1)}{f(j+1)+\alpha}.
\end{equation}
We analyse the convergence properties of $\mUnbiased(\alpha)$. Using $\log(1+x) \le x$ for $x \ge 0$ and $f(j)=\sum_{i=0}^{j-1}\triangle f(i)+f(0) \ge j\gamma$, we estimate
\begin{align*}
\prod_{j=0}^k \frac{f(j)}{f(j)+\alpha}&=\exp\Big(-\sum_{j=0}^k \log\Big(1+\frac{\alpha}{f(j)}\Big)\Big) \ge \exp\Big(-\sum_{j=0}^k \frac{\alpha}{f(j)}\Big) \\
&\ge \exp\Big(-\sum_{j=1}^k \frac{\alpha}{j \gamma}-\frac{\alpha}{f(0)}\Big) =\exp\Big(-\frac{\alpha}{\gamma} \log k - C-\delta_k\Big) 
\end{align*}
for some $C>0$ and a null sequence $(\delta_k)_{k \in \N}$. In particular, there exists a $C'>0$ such that for all $k \in \N$
\[
\prod_{j=0}^k \frac{f(j)}{f(j)+\alpha} \ge C' \left(\frac{1}{k}\right)^{\alpha/\gamma}.
\]
Hence, $\mUnbiased(\alpha)=\infty$ for all $\alpha \le \gamma$. On the other hand, by Cauchy's condensation test, $\mUnbiased(\alpha)<\infty$ for all $\alpha > \gamma$. 
Hence, $A_{\alpha}1(0) < \infty$ if and only if $\alpha \in (\gamma,1-\gamma)=:\interval$. From now on we assume that $\interval\not=\emptyset$, i.e.~$\gamma<\frac{1}{2}$. 
For $\alpha \in \interval$, $\rho(\alpha)$ is finite and since $\rho(\alpha) \ge \mUnbiased(\alpha)+\mUnbiased(1-\alpha)$, we see that $\rho(\alpha) \to \infty$ for $\alpha \to \partial \interval$. 
The existence and uniqueness of eigenfunction $\eigenfct_{\alpha}\colon \typespace \to (0,\infty)$ follows from the Krein-Rutman theorem (see Theorem~3.1.3 in \cite{Pin95}). 
The fact that $A_{\alpha}$ is a strictly positive operator, implies $\min_{\type \in \typespace}\eigenfct_{\alpha}(\type)>0$. 
Since $\rho(\alpha)$ is an isolated eigenvalue with one-dimensional eigenspace, one can argue along the lines of Chapter II $\S$3 (in particular Remark 2.4) and Theorem II.$\S$5.5.4 of \cite{Kat95} that $\rho$ is twice differentiable and the derivative can be represented as in \eqref{diff_eq}. 
In particular, $\rho''(\alpha)>0$ for all $\alpha \in \interval$, hence, $\rho$ is strictly convex on $\interval$ and there exists a unique minimizer $\alpha^* \in \interval$.

To prove (v), consider $\tau \geq 0$. 
Let $\hat Z^\ssup{\tau}_t= Z^\ssup{\tau} - \1_{[\tau, \infty)}$ for any $\tau \in [0,\infty)$.
Then, 
by the definition of the eigenfunction,
\[ \begin{aligned} \rho(\alpha) v_\alpha(\tau)  & = A_\alpha v_\alpha(\tau)
= E_{(0,\tau)} \Big[ \sum_{|x|=1} v_\alpha(\tau(x)) e^{-\alpha S(x)}
\Big] \\
& = E_{(0,\tau)} \Big[ \sum_{|x|=1} v_\alpha(- S(x)) e^{-\alpha S(x)} \1_{\{ S(x) \leq 0 \}} \Big]  +v_\alpha(\ell)  
E\Big[ \int_0^\infty e^{-\alpha t} \, d \hat Z^\ssup{\tau}_t \Big] \\
& \geq E_{(0,\ell)} \Big[ \sum_{|x|=1} v_\alpha(- S(x)) e^{-\alpha S(x)} \1_{\{ S(x) \leq 0 \}} \Big]  + 
v_\alpha(\ell)  
E\Big[ \int_0^\infty e^{-\alpha t} \, d \tilde Z_t \Big] \\
&  =  A_\alpha v_\alpha(\ell) = \rho(\alpha) v_\alpha(\ell) , \end{aligned} \]
where we used in the inequality that
 the distribution of the positions to the left of the origin do not depend on the initial type and
for the second expectation we used the monotonicity in types proved in~\cite{DerMoe13}.
The upper bound holds by a similar argument.
 \end{proof}

The next proposition collects some of the consequences for the spectral radius if we consider  converging attachment rules.

\begin{proposition}\label{prop:conv} Let $(f_t)_{t \geq 0}$ be  pointwise decreasing attachment rules with $\gamma_t = \inf_{k \geq 0} \frac{f_t(k)}{k} < 1/2$. Define
\[ f(k) := \lim_{t \ra \infty} f_t(k) . \]
Then, $f$  is concave, $f(0) \leq 1$ and $f(k+1) - f(k) \leq 1$ for all $k \in \N_0$. Let $\gamma = \inf_{k \geq 0 } \frac{f(k)}{k}$ and let $\rho(\alpha)$ be the spectral radius of the operator associated to the branching process with attachment rule $f$, $\alpha^*$ be its unique minimizers and let $v_\alpha$ be the corresponding eigenfunction with $\| v_\alpha\| = 1$.
Define the same quantities with index $t$ when referring to the branching process associated to $f_t$, where we set $\rho_t(\alpha) = \infty$ if $t \notin (\gamma_t, 1- \gamma_t)$.
\begin{itemize}
	\item[(i)]
\[ \rho(\alpha) = \lim_{ t \ra \infty} \rho_t(\alpha) \mbox{ for all } t \in (\gamma , 1 -\gamma). \]
\item[(ii) ] Suppose that $\lim_{t \ra \infty} \rho_t(\alpha_t^*) = 1$, then  $\alpha_t^* \ra \alpha^*$  as $t \ra\infty$ and $\rho(\alpha^*) = 1$;
\item[(iii)] The quotient
\[
\frac{v_{t,\alpha_t^*}(0)}{v_{t,\alpha_t^*}(\ell)}
\]
is uniformly bounded in $t$ away from zero and infinity.
\item[(iv)] Moreover, as $t \ra \infty$,
\[ \rho_t''(\alpha_t^*) \ra \rho''(\alpha^*) . \]
\end{itemize}
\end{proposition}

\begin{proof}
Since $(f_t)_{t \ge 0}$ is a pointwise decreasing sequence of positive functions 
\[
f(k):=\lim_{t \to \infty} f_t(k),\qquad \text{for all } k \in \N_0,
\]
exists. As an infimum of concave functions, $f\colon \N_0\to [0,\infty)$ is concave. The property $f(0)\le 1$ is inherited from $f_t$. The increments might in general only satisfy $f(k+1)-f(k) \le 1$, but the strict inequality is not needed for the analysis of the branching process and the corresponding operators.

The assumption $\rho_t(\alpha_t^*)>1$ implies that there exists a giant component for all $t\ge 0$. Hence,
\[
\mUnbiased_{f_t}(1/2)+\sqrt{\mUnbiased_{f_t}(1/2)\mBiased_{f_t}(1/2)}> 1 \quad \forall t \ge 0,
\]
by Proposition 1.10 in \cite{DerMoe13}. Here $a_{f_t}$ and $c_{f_t}$ are the functions given in \eqref{ac_fctns}. By monotone convergence, $\mUnbiased_{f}(1/2)+\sqrt{\mUnbiased_{f}(1/2)\mBiased_{f}(1/2)}\ge 1$ and, in particular, $f(0)>0$. Hence, $f$ is an attachment rule. From now on we add a subscript $t$ to all quantities corresponding to $f_t$ and no subscript for quantities corresponding to $f$.

The assumption $\gamma_t <\frac{1}{2}$ is needed to make the operator $A_{t,\alpha}$ exist for some $\alpha$. We have for all $t \le s$
\[
\gamma:=\lim_{k \to \infty} \frac{f(k)}{k} \le \lim_{k \to \infty} \frac{f_s(k)}{k} \le \lim_{k \to \infty} \frac{f_t(k)}{k} = \gamma_t <\frac{1}{2}.
\]
In particular, $(\gamma_t)_{t \ge 0}$ is a non-increasing sequence. Let $\interval=(\gamma,1-\gamma)$ and $\interval_t=(\gamma_t,1-\gamma_t)$ for $t \ge 0$. Then $\interval_t \subseteq \interval$ for all $t\ge 0$ and we write $\rho_t(\alpha)=\infty$ whenever $\alpha \not\in \interval_t$.

Let $\alpha \in \interval$. We use the monotonicity of the branching process in the attachment rule to derive
\begin{align*}
\rho(\alpha)&=\inf_{n \in \N}\|A_{\alpha}^n1\|^{\frac{1}{n}} = \inf_{n \in \N}\inf_{t \ge 0} \|A_{t,\alpha}^n1\|^{\frac{1}{n}}= \inf_{t \ge 0} \inf_{n \in \N}\|A_{t,\alpha}^n1\|^{\frac{1}{n}}\\
&= \inf_{t \ge 0} \rho_t(\alpha)=\lim_{t \to \infty}\rho_t(\alpha).
\end{align*}
In particular, for every $\alpha \in \interval$ there exists $t\ge 0$ such that $\rho_t(\alpha)<\infty$. Hence, $\bigcup_{t \ge 0}\interval_t =\interval$. Since $\rho(\alpha) \to \infty$ for $\alpha \to \partial \interval$, there exists an $\eps>0$ such that $\alpha^*,\alpha_t^* \in [\gamma+\eps,1-\gamma-\eps]=:\hat{\interval}$ for all $t \ge 0$. 

Convergence $\gamma_t \to \gamma$ implies the existence of a $t_0>0$ such that $\hat{\interval}\subseteq \interval_t$ for all $t\ge t_0$. In particular, we can consider the family $\rho,(\rho_t)_{t\ge t_0}$ of uniformly continuous functions on $\hat{\interval}$.

In the next step we argue that $\alpha_t^* \overset{t \to \infty}{\longrightarrow} \alpha^*$. Notice that by assumption $\rho(\alpha) =\lim_{t \to \infty}\rho_t(\alpha)\ge \lim_{t \to \infty} \rho_t(\alpha_t^*) \ge 1$ and 
\[
\rho(\alpha^*)\le \rho(\alpha_t^*) \le \rho_t(\alpha_t^*) \to 1.
\]
Hence, $\rho(\alpha^*)=1$. Suppose that $\alpha_t^* \overset{t \to \infty}{\longrightarrow} \alpha^*$ does not hold. Then there exists a $\delta>0$ and a subsequence $t_n\uparrow \infty$ such that $|\alpha_{t_n}^*-\alpha^*| \ge \delta$ for all $n \in \N$. Since $\rho$ is strictly convex with unique minimizer $\alpha^*$, we have
\[
\delta':=\min\{ \rho(\alpha^*-\delta), \rho(\alpha^*+\delta)\}-1 >0
\]
and $\rho(\alpha)\ge 1+\delta'$ for all $\alpha \not\in [\alpha^*-\delta,\alpha^*+\delta]$. In particular,
\[
\rho_{t_n}(\alpha^*) \ge \rho_{t_n}(\alpha_{t_n}^*)\ge \rho(\alpha_{t_n}^*) \ge 1+\delta'.
\]
Since the term on the left-hand side converges to $1$, this is a contradiction and the convergence $\alpha_t^* \overset{t \to \infty}{\longrightarrow} \alpha^*$ is established.

The fact that $\alpha_t^*$ converges to $\alpha^*$ and $f_t$ converges to $f$ implies that $A_{t,\alpha_t^*}$ converges to $A_{\alpha^*}$ because of the uniform continuity of the operator in $\alpha \in \hat{\interval}$. Since the eigenspaces of $\rho(\alpha^*)$ and $\rho_t(\alpha_t^*)$ are one dimensional, one can argue along the lines of Note 3 on Chapter II in \cite[pages 568-569]{Kat95} to see that $v_{t,\alpha_t^*}$ converges to $v_{\alpha^*}$. Since the functions are bounded,
\[
\frac{v_{t,\alpha_t^*}(0)}{v_{t,\alpha_t^*}(\ell)}
\]
is uniformly bounded in $t$ from zero and infinity. With the observed convergences and this uniform bound \eqref{diff_eq} now implies that also
\[
\rho_t''(\alpha_t^*) \to \rho''(\alpha^*), 
\]
as required.
\end{proof}

\subsection{The many-to-one Lemma}\label{sec:mto}
We first continue the analysis of the $\IBRW$. The following lemma is based on a spine construction which is known as Lyons' change of measure \cite{Ly94}. Recall the Ulam-Harris notation from Section~\ref{sec:BP_def}.

\begin{lemma}[{\bf{Many-to-one}}]\label{lem:mto}
For all $\alpha \in \interval$ there exists a probability measure $\P^{\alpha}$ on some measurable space, and a Markov process $((S_n,\type_n) \colon n \in \N_0;\P^{\alpha})$ with state space $\R \times \typespace$, such that for all $n \in \N$, $(s_0,\type_0) \in \R \times \typespace$ and $F\colon(\R\times \typespace)^n \to \R_+$ measurable
\begin{equation}\label{eq:mto}
\begin{split}
E_{(s_0,\type_0)}^p\Big[\sum_{|x|=n} e^{-\alpha (\loc(x_n)-\loc(x_0))} \rho^p(\alpha)^{-n} \frac{\eigenfct_{\alpha}(\type(x_n))}{\eigenfct_{\alpha}(\type(x_0))}& F(\loc(x_1),\type(x_1),\ldots,\loc(x_n),\type(x_n))\Big]\\
&=\E_{(s_0,\type_0)}^{\alpha}[F(S_1,\type_1,\ldots,S_n,\type_n)].
\end{split}
\end{equation}
Moreover, $\sigma((S_i,\type_i)\colon i\in [n])=\sigma(S_i\colon i\in [n])$.
\end{lemma}

 Note that it is easy to check that the distribution of $((S_n, \tau_n) \, : \, n \in \N_0; \p^\alpha)$ does not depend on the percolation parameter $p$.

\begin{proof}
Given a labelled tree $(\Gamma,L)$, with $L(x)=(\loc(x),\type(x))$, we can distinguish an ancestral line 
$\xi=(\xi_1,\xi_2,\ldots)$ which we call spine. In the space of labelled trees, we denote by $\F_n$ the $\sigma$-field generated by the first $n$ generations, $\F_n=\sigma((x,L(x))\colon |x|\le n)$. 
The analogue in the space of trees with spines is denoted by $\F_n^*$.

For every $(s_0,\type_0)\in \R\times \typespace$, the distribution of the $\IBRW$ started in $(s_0,\type_0)$ can be interpreted as a distribution $P_{(s_0,\type_0)}^{p}$ on the set of labelled trees. 
We extend this measure to the space of labelled trees with spines. Since $(s_0,\type_0)$ and $p$ will remain fixed throughout the proof, we omit it from the notation and write $P=P_{(s_0,\type_0)}^p$ for brevity.

Note that every $\F_n^*$-measurable function $g$ can be written as
\[
g(\Gamma,L,\xi)=\sum_{|x|=n} g_x(\Gamma,L) \mathbbm{1}_{\xi_n=x},
\]
for $\F_n$-measurable functions $g_x$ (see page 24 in \cite{Rob09}). We define $P_n^*$ to be the (non-probability) measure on $\F_n^*$ such that for all nonnegative $\F_n^*$-measurable functions $g$,
\[
\int g(\Gamma,L, \xi) \, dP_n^*= \int \sum_{|x|=n} g_x(\Gamma,L) \,dP|_{\F_n}.
\]
We now construct a new branching random walk under a new probability measure $\p^\alpha$.
The root has again label $L(\emptyset)=(\loc(\emptyset),\type(\emptyset))=(s_0,\type_0)$. A particle $\xi_n$ on the spine in generation $n$ with label $(\loc(\xi_n),\type(\xi_n))$ produces new offspring with distribution
\[
\frac{d\hat{\mathcal{L}}_{\type(\xi_n)}}{d\mathcal{L}_{\type(\xi_n)}}(\mu)=\frac{1}{\rho^p(\alpha)}\int e^{-\alpha (\loc(x)-\loc(\xi_n))}\frac{\eigenfct_{\alpha}(\type(x))}{\eigenfct_{\alpha}(\type(\xi_n))} \, \mu(dx)
\]
for all atomic measures $\mu$ on $\R\times \typespace$. Here $\mathcal{L}_{\sigma}$ denotes the offspring distribution for a particle of type $\sigma \in \typespace$ in the original process. 
The new spine particle $\xi_{n+1}$ in generation $n+1$ is chosen from the offspring  of $\xi_n$ by choosing an offspring $x$ with probability proportional to
\[
e^{-\alpha(\loc(x)-\loc(\xi_n))}\eigenfct_{\alpha}(\type(x)).
\]
Off the spine the new branching random walk behaves exactly as the original one. 
Then
\begin{align*}
\frac{d\P^{\alpha}|_{\F_n^*}}{dP_n^*}(\Gamma,L,\xi)&=\frac{d\P^{\alpha}|_{\F_{n-1}^*}}{dP_{n-1}^*}(\Gamma,L,\xi) \, \rho^p(\alpha)^{-1} \sum_{x\colon \parent(x)=\xi_{n-1}} e^{-\alpha (\loc(x)-\loc(\xi_{n-1}))}\frac{\eigenfct_{\alpha}(\type(x))}{\eigenfct_{\alpha}(\type(\xi_{n-1}))}\\
& \phantom{space more space more}\frac{e^{-\alpha(\loc(\xi_{n})-\loc(\xi_{n-1}))}\eigenfct_{\alpha}(\type(\xi_{n}))}{\sum_{x\colon \parent(x)=\xi_{n-1}} e^{-\alpha(\loc(x)-\loc(\xi_{n-1}))}\eigenfct_{\alpha}(\type(x))}\\
&=\frac{d\P^{\alpha}|_{\F_{n-1}^*}}{dP_{n-1}^*}(\Gamma,L,\xi) \rho^p(\alpha)^{-1} e^{-\alpha(\loc(\xi_{n})-\loc(\xi_{n-1}))}\frac{\eigenfct_{\alpha}(\type(\xi_{n}))}{\eigenfct_{\alpha}(\type(\xi_{n-1}))}\\
&=\prod_{k=1}^n  \Big( \rho^p(\alpha)^{-1} e^{-\alpha(\loc(\xi_{k})-\loc(\xi_{k-1}))}\frac{\eigenfct_{\alpha}(\type(\xi_{k}))}{\eigenfct_{\alpha}(\type(\xi_{k-1}))}\Big)\\
&= \rho^p(\alpha)^{-n} e^{-\alpha(\loc(\xi_n)-\loc(\xi_{0}))}\frac{\eigenfct_{\alpha}(\type(\xi_{n}))}{\eigenfct_{\alpha}(\type(\xi_{0}))}.
\end{align*}
In particular, for all $F \colon(\R \times \typespace)^n \to [0,\infty)$ measurable
\begin{align*}
\E^{\alpha}&\big[F(\loc(\xi_1),\type(\xi_1),\ldots,\loc(\xi_n),\type(\xi_n))\big]\\
&=E_n^{*}\Big[F(\loc(\xi_1),\type(\xi_1),\ldots,\loc(\xi_n),\type(\xi_n)) \rho(\alpha)^{-n} e^{-\alpha (\loc(\xi_n)-\loc(\xi_0))}\frac{\eigenfct_{\alpha}(\type(\xi_n))}{\eigenfct_{\alpha}(\type(\xi_0))}\Big]\\
&=E\Big[\sum_{|x| =n} F(\loc(x_1),\type(x_1),\ldots,\loc(x_n),\type(x_n)) \rho(\alpha)^{-n}e^{-\alpha (\loc(x_n)-\loc(x_0))}\frac{\eigenfct_{\alpha}(\type(x_n))}{\eigenfct_{\alpha}(\type(x_0))} \Big].
\end{align*}
We define $(S_n,\type_n):=(\loc(\xi_n),\type(\xi_n))$ and $\P_{(s_0,\type_0)}^{\alpha}=\P^{\alpha}$. The Markov property follows from the definition of the process. 
Since the offspring distribution of the spine is absolutely continuous with respect to the offspring distribution of the original process and the type of the original process is a function of the locations of its ancestors and the particle itself, the proof is complete.
\end{proof}

Next we also need a higher-dimensional version of the many-to-one lemma that includes the number of offspring of the particles on the spine.
For any $x = (x_0, \ldots, x_n)$ in the branching process, define the point measure on $\R$
\[ \nu_{x} = \sum_{y \, : \, |y|=n+1, y_n =  x} \delta_{S(y) - S(x)}  , \]
with $\delta_s$  a Dirac mass in $s$, which describes the positions of offspring of $x$ relative to the position of $x$.
Denote by $M_p(\R)$ the point measures on $\R$.

\begin{lemma}\label{le:mto_2} For all $\alpha \in \mathcal{I}$, $(s_0, \tau_0) \in \R\times \typespace$ there exists a probability measure $
\p^\alpha_{(s_0,\tau_0)} = \p^{\alpha,p}_{(s_0,\tau_0)}$ and a Markov process $((S_i, \tau_i, \nu_{i-1})_{i=0, \ldots, n}, : \, n \in \N_0; \p_{(s_0,\tau_0)}^\alpha)$ started in $(s_0, \tau_0, \delta_{s_0})$ with state space $\R \times \typespace\times M_p(\R)$ such that for all measurable $F$
\[ 
\begin{split}E_{(s_0,\type_0)}^p\Big[\sum_{|x|=n} e^{-\alpha (\loc(x_n)-\loc(x_0))} \rho(\alpha)^{-n} \frac{\eigenfct_{\alpha}(\tau(x_{n}))}{\eigenfct_{\alpha}(\type(x_0))}& F(\loc(x_1),\type(x_1),\nu_{x_0}, \ldots, 
\loc(x_n), \type(x_n), \nu_{x_{n-1}})\Big]\\
&=\E_{(s_0,\type_0)}^{\alpha}[F(S_1,\tau_1,\nu_0,\ldots,S_n,\tau_n, \nu_{n-1})].
\end{split}
\]
Moreover, for any measurable $A \subset \R$ and any measurable $F$,
\[\begin{aligned}  \E[ F(S_n-S_{n-1},&  \nu_{n-1}(A)) \, |\,  \sigma(S_{i}, \tau_{i},\nu_{i-1}, i\leq n-1) ] \\
& = \E_{(0,\tau)}\Big[ \sum_{|x|=1} e^{-\alpha S(x)} \rho(\alpha)^{-1} \frac{v_\alpha(\tau(x))}{v_\alpha(\tau)} F(S(x), \{ y \, | \, |y|=1, S(y) \in A\}) \Big]\Big|_{\tau=\tau_{n-1}} . \end{aligned} \]
\end{lemma}

 Note that unlike in Lemma~\ref{lem:mto} the distribution of $(S_k,\tau_k, \nu_{k-1})_{k \in \N_)}$ does depend on $p$, since we are considering a non-linear function of the point process describing the position of the offspring.

\begin{proof} Consider the IBRW with spine $\xi$ under the measure $\p^\alpha$ as constructed in the proof of Lemma~\ref{lem:mto}. Then, define $(S_n, \tau_n, \nu_{n-1}) := ( S(\xi_n), \tau(\xi_n), \nu_{\xi_{n-1}})$
and the first statement follows since we know the explicit Radon-Nikodym density of $\p^\alpha$ with respect to $P^*_n$. 
The second statement is a consequence of the Markov property of the IBRW with spine combined with a suitable choice of test function $F$.
\end{proof}

\begin{lemma}[{\bf{Moments}}]\label{moments_lem}
Let $\alpha \in \interval$ and $((\loc_n,\type_n)\colon n \in \N_0)$ be the Markov process from Lemma~\ref{lem:mto}. For all $\type_0 \in \typespace$
\[
\E^{\alpha}_{(0,\type_0)}[S_1]=-\frac{\rho'(\alpha)}{\rho(\alpha)} \qquad \E^{\alpha}_{(0,\type_0)}\big[S_1^2\big]=\frac{\rho''(\alpha)}{\rho(\alpha)}.
\]
\end{lemma}

\begin{proof}
Follows immediately from \eqref{eq:mto} and \eqref{diff_eq}.
\end{proof}

\subsection{Asymptotic moment estimates}\label{ssn:asymp_moments}

In the proofs of Theorems~\ref{thm:percol} and~\ref{thm:vary_f} we will need estimates for moments of 
the Markov chains defined in  Lemmas~\ref{lem:mto} and~\ref{le:mto_2}.
Suppose that $f_{t_n}$ is a sequence of decreasing attachment rules such that $f_{t_n} \downarrow f$ pointwise for an attachment rule $f$. Further, suppose that $\gamma_n := \lim_{k \ra \infty} f_{t_n}(k)/k < \frac 12$ for all $n \in\N$,
so that by Proposition~\ref{prop:conv} also $\gamma = \lim_{k \ra \infty} \frac{f(k)}{k} < \frac 12$.  Let $\alpha^*_n$, resp.\ $\alpha^*$ be the unique minimizer of the spectral radius $\rho_n$, resp.\ $\rho$, of the operator corresponding to the attachment rule $f_n$, resp.\ $f$.
In the setting of Theorem~\ref{thm:percol} take $f_{t_n} = f$ and write $\rho_n = \rho^{p_n} = p_n \rho(\cdot)$.

The Markov chain from Lemma~\ref{lem:mto} corresponding to attachment rule $f_{t_n}$ 
is denoted by $((S_i^\ssup{n},\type_i^\ssup{n})\colon i \in \N;\P=\P_{(s_0,\ell)}^{\alpha_n^*}$),

\begin{lemma}\label{le:exp_moments} 
There exists $\eta > 0$ such that
	\[ \sup_{n \in \N, \tau \in \typespace} \E_{0,\tau}^{\alpha_n^*} \Big[ e^{\eta |S_1^\ssup{n}|}\Big] < \infty. \]
\end{lemma}

\begin{proof} We first consider the case $S_1^\ssup{n} \geq 0$.
Note that 
\[\begin{aligned}  \E^{\alpha_n} \Big[ e^{\eta S_1^\ssup{n}}\Big] 
& \leq 1  + E_{(0,\tau)}\Big[ \sum_{|x|=1} e^{(\eta -  \alpha_n^\star) S^\ssup{n}(x)} \frac{v_{\alpha_n}(\tau(x))}{v_{\alpha_n}(\tau)} \rho_n(\alpha_n)^{-1} \1_{\{ S^\ssup{n}(x) \geq 0 \} }\Big] \\
& \leq 1  + \sup_{k \in \N} \frac{v_{\alpha_k}(0)}{v_{\alpha_k}(\ell)} \
E_{(0,\tau)}\Big[ \sum_{|x|=1} e^{(\eta -  \alpha_n^\star) S^\ssup{n}(x)}   \1_{\{ S^\ssup{n}(x) \geq 0 \} }\Big] 
\end{aligned} 
\]
where we used that $\rho_n(\alpha_n)>1$, the monotonicity in types for $v_{\alpha_n}$ by Lemma~\ref{lem:SpecProp} and the uniform boundedness of the quotient
$\frac{v_{\alpha_n}(0)}{v_{\alpha_n}(\ell)}$ in $n$ by Proposition~\ref{prop:conv}.
Hence, it suffices to show that for the right choice of $\eta$ the 
expectation on the right hand side remains bounded in $n$.

Now, by Lemma~\ref{lem:SpecProp}, we have that $\eta := \frac14 (\alpha^*- \gamma) > 0$.
By Proposition~\ref{prop:conv}, we have that $\gamma_n \downarrow \gamma$ and $\alpha_n^*\ra \alpha^*$. Then, we can choose $n_0$ sufficiently large such that for all $n \geq n_0$ we have
$\gamma_{n} < \gamma + \eta$ and $\alpha_n^\star \geq \alpha^* - \eta$.

Furthermore, we denote by $\tilde Z^\ssup{n_0}$ the jump process that jumps from $k$ to $k+1$ with rate $f_{t_{n_0}}(k)$ started in $1$. Then 
by comparison with a Yule process with constant branching rate, we can find a constant $C(n_0) > 0$ such that
\begin{equation}\label{eq:2610-1} \sup_{ t \geq 0} E[\tilde Z^\ssup{n_0}_t] e^{-(\gamma + \eta)t} \leq C(n_0). \end{equation}
Finally, we obtain from the construction of IBRW and using~\eqref{eq:2610-1} that for $n \geq n_0$ 
\[ \begin{aligned} E_{(0,\tau)}\Big[  \sum_{|x|=1} & e^{(\eta -  \alpha_n^\star) S^\ssup{n}(x)}   \1_{\{ S^\ssup{n}(x) \geq 0 \} }\Big]
 \leq \sum_{k \in \N_0} e^{- (\alpha^* - 2 \eta) k} \, E_{(0,\tau)} \Big[ \sum_{|x|=1} 
\1_{\{ S(x) \in [k,k+1]\}} \Big] \\
&  \leq \sum_{k \in \N_0}  e^{- (\alpha^* - 2 \eta) k} E [ \tilde Z_{k+1}^\ssup{n_0} ] 
\leq C(n_0) \sum_{k \in \N_0} e^{- (\alpha^*  - 3\eta - \gamma) k  + \gamma+ \eta},
\end{aligned}
\]
which is finite by choice of $\eta$.

The fact that the supremum over all $n$ is finite follows from  the same argument  if we redefine  $\eta$ as $\frac 14 \min\{\alpha^* - \gamma, 
\alpha^*_k - \gamma_k , k \leq n_0 \}$.

For the case that $S_1^\ssup{n} \leq 0$, it suffices to prove in a second step, that there exists $\eta > 0$ such that 
\[ E_{(0,\tau)}\Big[ \sum_{|x| =1} e^{-(\alpha_n^* +\eta) S^\ssup{n}(x) } \1_{\{S^\ssup{n}(x) \leq 0\}} \Big] < \infty . \]
Since the children to the left of a particle in the IBRW form a Poisson process and their distribution is not depending on the type of their ancestor we have by construction 
\[ E_{(0,\tau)}\Big[ \sum_{|x| =1} e^{-(\alpha_n^* +\eta) S^\ssup{n}(x) } \1_{\{S^\ssup{n}(x) \leq 0\}} \Big]
= \int_0^\infty e^{(\alpha_n^* + \eta) t} e^{-t} E[ f_{t_n} (X^\ssup{f_{t_n}}_t) ] \, dt , \]
where $X^\ssup{f}$ is the pure birth process with jump rates given by $f$.

Let $\eps_n = \frac 18 ( 1 - \gamma_n - \alpha^*_n)$ and let 
$\eps = \frac 14 ( 1 - \gamma - \alpha^*)$, where $\gamma = \inf_{k \geq 1} \frac{f(k)}{k}$. 
By Lemma~\ref{lem:SpecProp}, $\eps_n, \eps > 0$.

Using that $f_{t_k} (k) \leq k+1$ and a comparison to a Yule process, we have that there exists $C_n > 0$ such that
\[E[ f_{t_n} (X_t^\ssup{f_{t_n}}) ] \leq  E[ X_t^\ssup{f_{t_n}}] +1\leq C_n \, e^{(\gamma_n + \eps_n) t} \quad \mbox{for all } t \geq 0 . \]
By Proposition~\ref{prop:conv}, we can find $n_0$ such that for all $n \geq n_0$, 
\[ |\alpha_n - \alpha^*| < \eps, \quad \mbox{and} \quad \gamma \leq \gamma_n \leq \gamma + \eps. \]
Define $\eta := \min \{ \frac{3}{8} \eps ,  \eps_n, n \leq n_0 \}$.
Then, we have for $n \leq n_0$,
\[\int_0^\infty e^{(\alpha_n^* + \eta) t} e^{-t} E[ f_{t_n} (X^\ssup{f_{t_n}}) ] \, dt 
\leq C_n \int_0^\infty e^{(\alpha_n^* + \eta - 1 + \gamma_n + \eps_n ) t} \, dt 
\leq C_n \int_0^\infty e^{- 6 \eta t} \, dt < \infty. \]
Furthermore, for $n \geq n_0$, we can use the monotonicity of $f_{t_n}$ to deduce that 
\[\begin{aligned}  \int_0^\infty e^{(\alpha_n^* + \eta) t}  e^{-t} E[ f_{t_n} (X^\ssup{f_{t_n}}) ] \, dt
 &\leq \int_0^\infty e^{(\alpha_n^* + \eta) t} e^{-t} E[ f_{t_{n_0}} (X^\ssup{f_{t_{n_0}}}) ] \, dt \\ 
& \leq C_{n_0}  \int_0^\infty e^{(\alpha_n^* + \eta - 1 + \gamma_{n_0} + \eps_{n_0} ) t}  \, dt\\
& \leq C_{n_0}  \int_0^\infty e^{(\alpha^*   + \eta - 1 + \gamma + 2\eps +  \frac 18 ( 1 - \gamma - \alpha^* + \eps) ) t}  \, dt\\
& \leq C_{n_0}  \int_0^\infty e^{- \eps t} \, dt, 
\end{aligned} \]
which completes the second step and thus the proof of the lemma.
\end{proof}

\begin{lemma}\label{le:bound_children} 
Let $(S_k^\ssup{n}, \tau^\ssup{n}_k, \nu_{k-1}^\ssup{n})$ be the Markov chain defined in Lemma~\ref{le:mto_2}
either for attachment rule $f_{t_n}$ and percolation parameter $1$ or for fixed attachment rule $f$ (with $\gamma < \frac 12$) and percolation parameter $p_n$, where $p_n \downarrow \rho(\alpha^*)^{-1}$.
For any sequence $(M_n)_{n \in \N}$ such that $M_n \ra \infty$, there exist constants $C > 0$ and $\hat \gamma > 0$ such that, for all~$n$,
\[ \sup_{\tau \in \typespace} \E_{(0,\tau)}^{\alpha_n^*}[ \nu^\ssup{n}_0((-\infty, M_n)) \1_{\{S_1^\ssup{n} \leq M_n\}} ] \leq C e^{\hat \gamma M_n}. \]
\end{lemma}

\begin{proof} To unify notation define $p_n=1$ in the varying $f$ case and  $f_{t_n} = f$ in the percolation case.
By the extended many-to-one formula, Lemma~\ref{le:mto_2}, we have that
\[\begin{aligned}  \E_{(0,\tau)}^{\alpha_n^*}\big[  & \nu^\ssup{n}_0((-\infty, M_n)) \1_{\{S_1^\ssup{n} \leq M_n\}} \big] \\
&  =
E_{(0,\tau)}^{\, p_n} \bigg[ \Big(  \sum_{|x| =1} e^{-\alpha_n^* S^\ssup{n}(x) } \rho_n(\alpha_n^*)^{-1} \frac{v_{\alpha_n^*}(\tau(x))}{v_{\alpha_n^*}(\tau) } (\1_{\{ S^\ssup{n}(x) \leq M_n\}} \Big)\Big(  \sum_{|x| =1} \1_{\{ S^\ssup{n}(x) \leq M_n\}} \Big) \bigg] .
\end{aligned} \] 
We can use that $\rho_n(\alpha_n^*) > 1$, the monotonicity in types and in $p_n$, Lemma~\ref{lem:SpecProp}, 
and that $C :=\sup_{n\in \N} \frac{v_{\alpha_n}(0)}{v_{\alpha_n}(\ell) } < \infty$ by Proposition~\ref{prop:conv},
in order to bound the above by
\begin{equation}\label{eq:2111-1} \begin{aligned}  
E_{(0,\tau)}^1 \bigg[ & \Big(  \sum_{|x| =1} e^{-\alpha_n^* S^\ssup{n}(x) }  \frac{v_{\alpha_n}(0)}{v_{\alpha_n}(\ell) } (\1_{\{ S^\ssup{n}(x) \leq M_n\}} \Big)\Big(  \sum_{|x| =1} \1_{\{ S^\ssup{n}(x) \leq M_n\}} \Big) \bigg] \\
& \leq  C \, E_{(0,\tau)}^1 \Big[ \Big(  \sum_{|x| =1} \1_{\{ 0 \leq S^\ssup{n}(x) \leq M_n\}}  \Big)^2 \Big] \\
& \qquad + 
2 C \, E_{(0,\tau)}^1 \Big[ \Big(  \sum_{|x| =1}   e^{-\alpha_n^* S^\ssup{n}(x) }\1_{\{ S^\ssup{n}(x) \leq 0\}}\Big) \Big]
E_{(0,\tau)}^1 \Big[ \Big(  \sum_{|x| =1} \1_{\{ 0 \leq S^\ssup{n}(x) \leq M_n\}}  \Big) \Big] \\
&  \qquad + 
 C \, E_{(0,\tau)}^1 \bigg[ \Big(  \sum_{|x| =1}   e^{-\alpha_n^* S^\ssup{n}(x) }\1_{\{ S^\ssup{n}(x) \leq 0\}}\Big)
\Big(\sum_{|x|=1} \1_{\{ S^\ssup{n}(x) \leq 0 \}} \Big)  \bigg] . 
\end{aligned}
\end{equation}
For the first term in~\eqref{eq:2111-1}, we note that $f(k) \leq f_{t_n}(k) \leq f_{t_1}(k) \leq \gamma_1^+ k + f_{t_1}(0)$ (for $\gamma_1^+ = \sup \{ \Delta f_{t_1} (k), k \in \N_0\}$.
Therefore, if we let $(\hat Z^{f_{t_1}})_{t \geq 0}
$ be the jump process jumping from $k$ to $k+1$ at rate $f_{t_1}(k)$ started in~$1$, we can conclude that by construction of the IBRW
\[  E_{(0,\tau)}^1 \bigg[ \Big(  \sum_{|x| =1} \1_{\{ 0 \leq S^\ssup{n}(x) \leq M_n\}}  \Big)^2 \bigg]
\leq E \big[ \big(\hat Z^{f_{t_1}}_{M_n}\big)^2\big] \leq \tilde C(f_{t_1}) e^{\hat \gamma M_n}, \]
for some constants $C(f_{t_1})$ and $\hat \gamma > 0$, where the latter bound follows by comparison with a Yule process, whose second moments grow at most exponentially.

For the second term on the right hand side in~\eqref{eq:2111-1}, the first expectation is bounded uniformly in 
$n$ by (the second part of) Lemma~\ref{le:exp_moments} and the second expectation can be bounded by the second moment, so that the first 
part of the argument applies. 

For the final term in~\eqref{eq:2111-1}, we use that the particles to the left form a Poisson process, 
so that we can use a standard identity for Poisson processes, see e.g.~\cite[Equation (4.26)]{LastPenrose17}, to 
deduce that 
\[ \begin{aligned} E_{(0,\tau)} \bigg[ \Big(  \sum_{|x| =1}   e^{-\alpha_n^* S^\ssup{n}(x) }\1_{\{ S^\ssup{n}(x) \leq 0\}}\Big) &
\Big(\sum_{|x|=1} \1_{\{ S^\ssup{n}(x) \leq 0 \}} \Big)  \bigg]\\
&= \int_0^\infty e^{\alpha_n^*t} e^{-t} E[ f_{t_n} (X^\ssup{f_{t_n}}) ] \, dt
 \Big( 1 + \int_0^\infty  e^{-t} E[ f_{t_n} (X^\ssup{f_{t_n}}) ] \, dt \Big) . \end{aligned}\]
However, as in the proof of Lemma~\ref{le:exp_moments} the right hand side is bounded uniformly in $n$.
\end{proof}

\begin{corollary}\label{cor:negligible} In the setting of Lemma~\ref{le:bound_children}, we have for any sequence $N_n \ra \infty$ and with $R_n = e^{N^{1/4}}, M_n = N^{1/5}$ that there exist $\tilde C, \tilde \gamma > 0$ such that
\[ \inf_{\tau \in \typespace} \p_{(0,\tau)}^{\alpha_n^*} (  \bar\nu^\ssup{n}_{0} \leq R_n,  S^\ssup{n}_1 \leq M_n ) \geq 1 - \tilde C e^{-\tilde \gamma N^{1/5}}    .\]
\end{corollary}

\begin{proof}
By Lemma~\ref{le:exp_moments}, there exists $\eta > 0$ such that
$C := \sup_{n\in \N_0, \tau} \E_{(0,\tau)}^{\alpha_n^*}[ e^{\eta S_1^\ssup{n}}] < \infty$. 
Moreover, by Lemma~\ref{le:bound_children}, there exist $ C' > 0,  \gamma' > 0$ such that
$\sup_{\tau} \E^\alpha_{(0,\tau)}[ \bar\nu_{0},  S_1 \leq M_n] \leq  C' e^{ \gamma' M_n}$. 
Hence, we can estimate by Chebyshev's inequality
\[\begin{aligned} 
\sup_{\tau \in \typespace} \p_{(0,\tau)}^{\alpha_n^*}\big( (\bar\nu_{0} & \leq R_n,  S_1 \leq M_n)^c \big) \\
& \leq 
P^\alpha_{(0,\tau_{k-1})}( S_1 \geq M_n )
+P^\alpha_{(0,\tau_{k-1})} (  \bar\nu_{0} \geq R_n ,  S_1 \leq M_n)\\
& \leq e^{-\eta M_n} + \frac{1}{R_n}  C' e^{ \gamma' M_n} 
=  e^{-\eta N^{1/5}} +  C' e^{\gamma' N^{1/5} - N^{1/4}}, 
\end{aligned} \]
 by our choice of $M_n$ and $R_n$.
Therefore, the statement of the corollary follows by choosing $\tilde \gamma$ and $\tilde C$ appropriately.
\end{proof}

\subsection{Mogulskii's theorem}\label{sec:mogulskii}

The main technical tool in the proof of our main result is the following large deviation result 
due to Mogulskii in its original form. We state it here in a version adapted to Markov chains as a generalisation
of the version for random walks found  in \cite{GHS11}.

\begin{theorem}[{\bf{\cite{Mo74},\cite{GHS11}}}]\label{thm:mogulskii}
Let $\typespace$ be a nonempty set. We assume the following:
\begin{itemize}
\item[{\rm (i)}] For each $n \in \N$, $((\loc_i^{\sss(n)},\type_i^{\sss(n)})\colon i \in \N)$ is a Markov chain with values in $\R \times \typespace$. 
\item[{\rm (ii)}] $(a_n)_{n \in \N}$ and  $((k_n)_{n \in \N})$ are positive sequences with $a_n, k_n \to \infty$ and $a_n^2/k_n\to 0$ as $n \to \infty$. 
\item[{\rm (iii)}] For all $c_1,c_2,c_3,c_4 \in \R$, $A>0$ and for $r_n:=\lfloor Aa_n^2\rfloor$,
\begin{equation}\label{eq:Donsker}\begin{aligned}
\P_{(0,\type)}\Big(c_1 \le \frac{S_i^{\sss (n)}}{a_n} & \le c_2 \, \forall\, i \in [r_n];c_3 \leq \frac{S_{r_n}^{\sss (n)}}{a_n} \le c_4\Big) \\
& \xrightarrow{n\to \infty} P(c_1 \le \sqrt{\sigma^2 A}W_t \le c_2\, \forall\, t \in [0,1]; c_3 \leq 
\sqrt{\sigma^2 A}W_1 \le c_4)\end{aligned} 
\end{equation}
uniformly in $\type \in \typespace$, where $\sigma^2>0$ is a constant independent of $A$ and $(W_t\colon t \ge 0;P)$ is a standard Brownian motion.
\end{itemize}
 Let $g_1 < g_2$ be two continuous functions on $[0,1]$ with $g_1(0)\le 0\le g_2(0)$ and denote
\[
E_n:=\Big\{g_1\Big(\frac{i}{k_n}\Big) \le \frac{S_i^{\sss(n)}}{a_n} \le g_2\Big(\frac{i}{k_n}\Big) \; \forall\, i\in [k_n]\Big\}.
\]
Then, for all $\type_0 \in \typespace$,
\[
\lim_{n \to \infty} \frac{a_n^2}{k_n}\log P_{(0,\type_0)}(E_n) =-\frac{\pi^2 \sigma^2}{2} \int_0^1 \frac{1}{[g_2(t)-g_1(t)]^2} \; dt.
\]
Moreover, for any $b > 0$, $\tau_0 \in \mathcal{T}$,
\[
\lim_{n \to \infty} \frac{a_n^2}{k_n}\log P_{(0,\type_0)}(E_n, S_{k_n}^\ssup{n} \geq (g_2(1) - b)a_n) =-\frac{\pi^2 \sigma^2}{2} \int_0^1 \frac{1}{[g_2(t)-g_1(t)]^2} \; dt.
\]
\end{theorem}

The proof of Theorem~\ref{thm:mogulskii} is postponed to Appendix~\ref{mogulskii_proof}. 

 We will now show that we can apply Mogulskii's results to the Markov chains from Lemmas~\ref{lem:mto} and~\ref{le:mto_2}. We will treat both the setting of Theorem~\ref{thm:percol} and~\ref{thm:vary_f} at the same time 
and so continue using the notation introduced at the beginning of Section~\ref{ssn:asymp_moments}.

To this end, we first recall Donsker's theorem for martingale difference arrays (see for example Theorem~7.7.3 in \cite{Du91} or Theorem~18.2 in \cite{Bi99}). 
The theorem is usually stated for $r_n=n$, but it is straightforward to generalize the statement to the following:

\begin{proposition}\label{Donsker_prop}
Let $(r_n)_{n \in \N} \in \N^{\N}$ be a sequence with $r_n \uparrow \infty$ as $n \to \infty$. For every $n\in \N$, let $(\xi_i^n\colon 1 \le i \le r_n)$ be a family of random variables and denote $\F_i^n=\sigma(\xi_1^n,\dots,\xi_i^n)$ for all $i \le r_n$. Assume that
\begin{itemize}
\item[{\rm (i)}] $E[\xi_i^n|\F_{i-1}^n]=0$ for all $i \le r_n$, $n \in \N$.
\item[{\rm (ii)}] For all $t \in [0,1]$, $\sum_{i \le \lfloor tr_n\rfloor} E[(\xi_i^n)^2|\F_{i-1}^n] \to t$ in probability as $n \to \infty$.
\item[{\rm (iii)}] For all $\eps>0$, $\sum_{i \le r_n} E[(\xi_i^n)^2 \mathbbm{1}_{|\xi_i^n|>\eps}|\F_{i-1}^n] \to 0$ in probability as $n \to \infty$.
\end{itemize}
Then the linear interpolation of $(\sum_{i \le m} \xi_i^n\colon m \le r_n)$ converges weakly to a standard Brownian motion on $[0,1]$.
\end{proposition}

We use Donsker's theorem as follows.

\begin{lemma}\label{lem:DonskCondCheck}
Let $A>0$, $(a_n)_{n \in \N}$ be a positive sequence with $\lim_{n \to \infty} a_n =\infty$
and write $r_n=\lfloor Aa_n^2\rfloor$ for each $n \in \N$. Moreover, let $(S_i^\ssup{n}, \tau_i^\ssup{n})_{i \in \N_0}$ be the Markov chain introduced in Lemma~\ref{lem:mto} for attachment rule $f_{t_n}$.
For all $c_1,c_2, c_3, c_4 \in \R$,
\[
\P_{(0,\type)}\Big(c_1 \le \frac{S_i^\ssup{n}}{a_n} \le c_2 \, \forall\, i \in [r_n]; c_3 \le \frac{S_{r_n}^\ssup{n}}{a_n} \le c_4\Big) \to P(c_1 \le \sqrt{\sigma^2 A}W_t \le c_2\, \forall\, t \in [0,1]; c_3 \leq \sqrt{\sigma^2 A}W_1 \le c_4)
\]
as $n \to \infty$, uniformly in $\type \in \typespace$, where $(W_t\colon t \ge 0;P)$ is a standard Brownian motion.
\end{lemma}

\begin{proof} The first step is to 
show that the conditions of Proposition~\ref{Donsker_prop} are satisfied by the random variables
\begin{equation}\label{def_ximn}
\xi_i^n:=\frac{S_i^\ssup{n}-S_{i-1}^\ssup{n}}{\sqrt{r_n \sigma^2}},\qquad i=1,\ldots,r_n.
\end{equation}
Let $\F_i^{n}=\sigma(S_1^\ssup{n},\ldots,S_i^\ssup{n})$. For all $i \in [r_n]$,
\begin{align*}
\E&[\xi_i^{n}|\F_{i-1}^n]=\frac{1}{\sqrt{r_n \sigma^2}} \E_{(\loc_{i-1}^\ssup{n},\type^\ssup{n}_{i-1})}^{\alpha_n^*}[\loc_1^\ssup{n}-\loc_0^\ssup{n}]= \frac{1}{\sqrt{r_n \sigma^2}} \E_{(0,\type^\ssup{n}_{i-1})}^{\alpha_n^*}[\loc_1^\ssup{n}]=0
\end{align*}
by Lemma~\ref{moments_lem} and $\rho_n'(\alpha_n^*)=0$ as $\alpha_n^*$ minimizes $\rho_n$. Moreover, as $n \to \infty$,
\begin{align*}
\sum_{i \le \lfloor t r_n\rfloor } \E[(\xi_i^n)^2|\F_{i-1}^n]=\sum_{i \le \lfloor t r_n\rfloor} \frac{1}{r_n \sigma^2} \E_{(\loc_{i-1}^\ssup{n},\type^\ssup{n}_{i-1})}^{\alpha_n^*}\big[(\loc_1^\ssup{n}-\loc_0^\ssup{n})^2\big]=\frac{\lfloor t r_n\rfloor}{r_n \sigma^2} \frac{\rho_n''(\alpha_n^*)}{\rho_n(\alpha_n^*)} \to \frac{t}{\sigma^2} \sigma^2=t,
\end{align*}
since $r_n \to \infty$ and $\rho_n(\alpha_n^*)/\rho_n''(\alpha_n^*)\to \sigma^2$ by definition. Thus, Condition~(ii) of Proposition~\ref{Donsker_prop} is satisfied. Finally, let $\eta$ be as in Lemma~\ref{le:exp_moments}, then
\begin{align*}
\sum_{i \le r_n} \E\big[(\xi_i^n)^2 \mathbbm{1}_{|\xi_i^n|>\eps}|\F_{i-1}^n\big] &=\frac{1}{r_n \sigma^2} \sum_{i \le r_n} \E_{(0,\type^\ssup{n}_{i-1})}^{\alpha_n^*}\big[(\loc_1^\ssup{n})^2 \mathbbm{1}\big\{|\loc_1^\ssup{n}|>\eps \sqrt{r_n \sigma^2}\big\}\big]\\
&\le \sigma^{-2} \sup_{\type \in \typespace} \E_{(0,\type)}^{\alpha_n^*}\big[(\loc_1^\ssup{n})^2 \mathbbm{1}\big\{|\loc_1^\ssup{n}|>\eps \sqrt{r_n \sigma^2}\big\}\big]\\
& \leq C \sigma^{-2} \frac{1}{\eps \sqrt{r_n \sigma^2}} \sup_{\type \in \typespace} \E_{(0,\type)}^{\alpha_n^*}\big[e^{\eta |(\loc_1^\ssup{n})|} \big] , 
\end{align*}
where $C$ is a constant such that $x^2 \leq C e^{\eta|x|}$ for all $ x\in \R$.
Then, by Lemma~\ref{le:exp_moments} the exponential moment is bounded uniformly in $n$, so that 
the right hand  side converges to zero as required.

To apply Theorem~\ref{thm:mogulskii} we have to check that the convergence of the distribution functions is uniform in the start type. 
This is guaranteed by the monotonicity of the $\IBRW$ in the start type (which was proven in \cite[Remark 2.6]{DerMoe13}) which entails a monotonicity of $(S_i^\ssup{n})$ by the many-to-one lemma, and by the fact that the limit is independent of the start type.
\end{proof}

\begin{lemma}\label{lem:DonskCondCheck2} Let $N_n$ be a positive sequence with $\lim_{n \ra \infty} N_n = \infty$ and set 
$a_n = N_n^{1/3}, M_n = N_n^{1/5}, R_n = e^{N^{1/4}}$.
Consider the  Markov chain $(\tilde S^\ssup{n}_k, \tilde \tau^\ssup{n}_k, \tilde \nu^\ssup{n}_{k-1})_{k \in \N_0}$  with filtration $(\mathcal F_k^\ssup{n})_{k \in \N_0}$ and transitions given for any measurable $F$ by 
\[ \E [ F(\tilde S_k^\ssup{n} - \tilde S^\ssup{n}_{k-1}, \tilde \tau^\ssup{n}_{\tau_k} , \tilde \nu^\ssup{n}_{k-1} )\, |\,
\mathcal{F}_{k-1}^\ssup{n} ] 
= \E^{\alpha_n}_{(0,\tau_{k-1})} [ F(S_1^\ssup{n}, \tau^\ssup{n}_1, \nu^\ssup{n}_0) \, |\, \nu^\ssup{n}_0((-\infty, M_n) ) \leq R_n, S^\ssup{n}_1\leq M_n ] .
\]
where $(S_1^\ssup{n}, \tau^\ssup{n}_1, \nu^\ssup{n}_0)$ is the first step of the  Markov chain 
defined in Lemma~\ref{le:mto_2} associated either to the IBRW with attachment rule $f_{t_n}$ and $p=1$
or with attachment rule $f$ and percolation parameter $p_n$.
Let $A>0$, 
 write $r_n=\lfloor Aa_n^2\rfloor$. For all $c_1,c_2, c_3,c_4 \in \R$,
\begin{equation}\label{eq:2411-1}\begin{aligned} 
\P_{(0,\type)}\Big(c_1 \le \frac{\tilde S_i^\ssup{n}}{a_n} & \le c_2 \, \forall\, i \in [r_n]; c_3 \le \frac{\tilde S_{r_n}^\ssup{n}}{a_n} \le c_4\Big) \\
 & \to P(c_1 \le \sqrt{\sigma^2 A}W_t \le c_2\, \forall\, t \in [0,1]; c_3 \leq \sqrt{\sigma^2 A}W_1 \le c_4)
\end{aligned}
\end{equation}
as $n \to \infty$, uniformly in $\type \in \typespace$, where $(W_t\colon t \ge 0;P)$ is a standard Brownian motion.
\end{lemma}

\begin{proof} We show that we can replace $\tilde S_i^\ssup{n}$ by $S_i^\ssup{n}$ in the above probability up to an error that converges to $0$ uniformly in $\tau$ and then invoke Lemma~\ref{lem:DonskCondCheck}. Define the event
\[ E_n := \{\bar \nu_0^\ssup{n} \leq R_n , S_1^\ssup{n} - S_0^\ssup{n} \leq M_n \}. \]
and note that by Corollary~\ref{cor:negligible}, there exist constants $ \tilde C, \tilde \gamma > 0$ such that 
\begin{equation}\label{eq:2311-1} \inf_{\tau \in \typespace} \p_{(0,\tau)} ( E_n) \geq 1 - \tilde C e^{-\tilde \gamma N^{1/5}}    .\end{equation}
We estimate for any $c^i_1, c^i_2 \in \R$
\[\begin{aligned}  \p_{(0,\tau)}^{\alpha_n^*}\Big( c_1^i \leq & \frac{ \tilde S_i^\ssup{n}}{a_n} \leq c_2^i \, \forall i \in [r_n] \Big) \\
& \leq  \E_{(0,\tau)}^{\alpha_n^*} \bigg[ \1{ \{ c_1^i \leq \frac{ \tilde S_i^\ssup{n}}{a_n} \leq c_2^i \, \forall i \in [r_n-1 ]\}} 
\frac{ \p_{(\tilde S_{r_n-1}, \tilde \tau_{r_n -1} )}^{\alpha_n^*} ( c_1^{r_n} \leq S_1^\ssup{n} \leq c_2^{r_n} ) }{ \p_{(0,\tilde \tau_{r_n -1})}^{\alpha_n^*}(E_n)} \bigg] \\
& \leq  \E_{(0,\tau)}^{\alpha_n^*} \Big[ \1{ \{ c_1^i \leq \frac{ \tilde S_i^\ssup{n}}{a_n} \leq c_2^i \, \forall i \in [r_n-1 ]\}} 
 \p_{(\tilde S_{r_n-1}, \tilde \tau_{r_n -1} )}^{\alpha_n^*} ( c_1^{r_n} \leq S_1^\ssup{n} \leq c_2^{r_n} )  \Big] ( 1  + 2 \tilde C e^{-\tilde \gamma N^{1/5}} ), 
\end{aligned} \]
where we assume that $n$ is sufficiently large and we used~\eqref{eq:2311-1} together with  $\frac{1}{1-x} \leq 1 +2x$ for $x \in [0,\frac 12]$.
Iterating this estimate yields
\begin{equation}\label{eq:2911-1}\begin{aligned}  \p_{(0,\tau)}^{\alpha_n^*} \Big( c_1^i \leq & \frac{ \tilde S_i^\ssup{n}}{a_n} \leq c_2^i \, \forall i \in [r_n] \Big) 
& \leq \P_{(0,\type)}\Big(c_1^i \le \frac{ S_i^n}{a_n} \le c_2^i \, \forall\, i \in [r_n]\Big) (1 + 2 \tilde C e^{-\tilde \gamma N^{1/5}})^{r_n} . \end{aligned}  \end{equation}
and we note that by the choice of $r_n$ the error converges to $0$.

For a lower bound, we estimate
\[\begin{aligned}  \p_{(0,\tau)}^{\alpha_n^*} \Big( c_1^i \leq & \frac{ \tilde S_i^\ssup{n}}{a_n} \leq c_2^i \, \forall i \in [r_n] \Big) \\
&  \geq  \E_{(0,\tau)}^{\alpha_n^*} \Big[ \1{ \{ c_1^i \leq \frac{ \tilde S_i^\ssup{n}}{a_n} \leq c_2^i  \, \forall i \in [r_n-1 ]\}} 
 \p_{(\tilde S_{r_n-1}, \tilde \tau_{r_n -1} )}^{\alpha_n^*} ( c_1^{r_n} \leq S_1^\ssup{n} \leq c_2^{r_n} , E_n) \Big] \\
&  \geq  \E_{(0,\tau)}^{\alpha_n^*} \Big[ \1{ \{ c_1^i \leq \frac{ \tilde S_i^\ssup{n}}{a_n} \leq c_2^i \,  \forall i \in [r_n-1 ]\}} 
 \p_{(\tilde S_{r_n-1}, \tilde \tau_{r_n -1} )}^{\alpha_n^*} ( c_1^{r_n} \leq S_1^\ssup{n} \leq c_2^{r_n} ) \Big]- \sup_{\tau} \p_{(0,\tau)}^{\alpha_n^*} (E_n)   
\end{aligned} \]
Iterating and using the bound~\eqref{eq:2311-1} gives
\begin{equation}\label{eq:2911-2}  \p_{(0,\tau)}^{\alpha_n^*} \Big( c_1^i \leq \frac{ \tilde S_i^\ssup{n}}{a_n} \leq c_2^i \forall i \in [r_n] \Big)
\geq \p_{(0,\tau)}^{\alpha_n^*} \Big( c_1^i \leq \frac{ \tilde S_i^\ssup{n}}{a_n} \leq c_2^i \forall i \in [r_n] \Big)
- \tilde C \, r_n \, e^{-\tilde \gamma N^{1/5}} , \end{equation}
where the error converges to $0$ by choice of $r_n$.
Then, combining~\eqref{eq:2911-1} and~\eqref{eq:2911-2} 
together with Lemma~\ref{lem:DonskCondCheck} we can deduce the 
statement of the lemma.
\end{proof}

\section{Proofs: Upper bound}\label{sec:upper_bound}

In this section, we fix the start type $\ell$ of the $\IBRW$. Recall that in the killed $\IBRW$ every particle $x$ with $\loc(x)>0$ is deleted together with its descendants. We denote its survival probability by $\zeta$, that is, for $s_0 \leq 0$,
\[
\zeta_{s_0}(p,f)=P_{(s_0,\ell)}^p(\text{killed $\IBRW$ survives}).
\]

\begin{lemma}\label{lem:pre_upper}
For all $\alpha \in \interval$, $s_0 \leq 0$, $n \in \N$ and $b_1,\dots, b_n \ge 0$, 
\[
\zeta_{s_0}(p,f) \le e^{-\alpha s_0}{ \rho^p}(\alpha)^n I(n)+\sum_{j=0}^{n-1} e^{-\alpha (b_{j+1}+s_0)}\rho^p(\alpha)^{j+1}I(j),
\]
where $I(j)=\P_{(s_0,\ell)}^{\alpha}(-b_i \le S_i \le 0 \;\forall \,i\in [j])$ for $j=0,\dots,n$,  where $((S_i, \tau_i)_{i \geq 0}, \p_{(s_0,\ell)})$ is the Markov chain from Lemma~\ref{lem:mto} started in $(s_0,\ell)$.
\end{lemma}

\begin{proof} By definition,
\begin{align}
\zeta_{s_0}(p,f) &\le P_{(s_0,\ell)}^p(\exists |x|=n\colon \loc(x_i) \le 0 \; \forall\, i\in [n])\notag \\
&\le P_{(s_0,\ell)}^p(\exists |x|=n\colon -b_i \le \loc(x_i) \le 0 \; \forall\, i\in [n]) \label{eq:preUp1}\\
&\phantom{space more}+\sum_{j=1}^{n} P_{(s_0,\ell)}^p(\exists |x|=n\colon -b_i \le \loc(x_i) \le 0 \; \forall\, i\in [j-1], \loc(x_j) < -b_j).\notag 
\end{align}
For the first summand we use the Markov inequality and Lemma~\ref{lem:mto} to derive
\begin{align}
P_{(s_0,\ell)}^p(\exists |x|=n\colon& -b_i \le \loc(x_i) \le 0 \; \forall\, i\in [n])\notag \\
&=P_{(s_0,\ell)}^p\Big(\sum_{|x|=n} \mathbbm{1}\{-b_i \le \loc(x_i) \le 0 \; \forall \,i \in [n]\} \ge 1\Big)\notag  \\
&\le E_{(s_0,\ell)}^p\Big[\sum_{|x|=n} \mathbbm{1}\{-b_i \le \loc(x_i) \le 0 \; \forall\, i \in [n]\}\Big]\notag  \\
&=\rho^p(\alpha)^n \E_{(s_0,\ell)}^{\alpha}\Big[e^{\alpha (S_n-S_0)} \mathbbm{1}\{-b_i \le S_i \le 0 \; \forall\, i \in [n]\}\frac{\eigenfct_{\alpha}(\type_0)}{\eigenfct_{\alpha}(\type_n)}\Big]\notag \\
&\le\frac{\eigenfct_{\alpha}(\ell)}{\inf_{\type \in \typespace} \eigenfct_{\alpha}(\type)} \rho^p(\alpha)^n e^{-\alpha s_0} I(n) = \rho^p(\alpha)^n e^{-\alpha s_0} I(n), \label{eq:preUp2}
\end{align}
 where we used in the last step that by Lemma~\ref{lem:SpecProp} 
$\inf_{\type \in \typespace} \eigenfct^p_{\alpha}(\type) = \eigenfct_{\alpha}(\ell)$.
Similarly,
\begin{align}
P_{(s_0,\ell)}^p(\exists |x|=n\colon& -b_i \le \loc(x_i) \le 0 \; \forall\, i\in [j-1], \loc(x_j) < -b_j)\notag \\
&\le P_{(s_0,\ell)}^p(\exists |x|=j\colon -b_i \le \loc(x_i) \le 0 \; \forall\, i \in [j-1], \loc(x_j) < -b_j) \notag\\
&=P_{(s_0,\ell)}^p\Big(\sum_{|x|=j}\mathbbm{1}\big\{-b_i \le \loc(x_i) \le 0 \; \forall\, i \in [j-1], \loc(x_j) < -b_j\big\}\ge 1\Big)\notag \\
& \le E_{(s_0,\ell)}^p\Big[\sum_{|x|=j}\mathbbm{1}\big\{-b_i \le \loc(x_i) \le 0 \; \forall\, i \in [j-1], \loc(x_j) < -b_j\big\}\Big]\notag \\
& = \rho^p(\alpha)^j\E_{(s_0,\ell)}^{\alpha}\Big[e^{\alpha (S_j-S_0)} \mathbbm{1}\big\{-b_i \le S_i \le 0 \; \forall \,i \in [j-1], S_j < -b_j\big\}\frac{\eigenfct_{\alpha}(\type_0)}{\eigenfct_{\alpha}(\type_j)}\Big]\notag \\
& \le 
 \rho^p(\alpha)^j e^{\alpha (-b_j-s_0)} I(j-1).\label{eq:preUp3}
\end{align}
Combining \eqref{eq:preUp1}--\eqref{eq:preUp3}, now concludes the proof.
\end{proof}

\begin{lemma}\label{lem:upper_estim}
For all $\alpha \in \interval$ with $\rho(\alpha) \ge 1$ and for all $s_0 \leq 0, k,N \in \N$ and $(b_j \colon j \in [kN])$ nonnegative and decreasing,
\[
\zeta_{s_0}(p,f) \le e^{-\alpha s_0}\rho(\alpha)^{kN} I(kN) +k \sum_{l=0}^{N-1} e^{-\alpha (b_{(l+1)k} +s_0)} \rho(\alpha)^{(l+1)k} I(lk).
\]
\end{lemma}

\begin{proof}
Using Lemma~\ref{lem:pre_upper}, that $j \mapsto I(j)$ is decreasing, $\rho(\alpha)\ge 1$ and that $j \mapsto e^{-\alpha b_j}$ is increasing to obtain
\begin{align*}
\zeta(p,f)- e^{-\alpha s_0}\rho(\alpha)^{kN} I(kN)  &\le \sum_{j=0}^{kN-1} e^{-\alpha (b_{j+1}+s_0)}\rho(\alpha)^{j+1}I(j)\\
&=\sum_{l=0}^{N-1} \sum_{j=lk}^{(l+1)k-1} e^{-\alpha (b_{j+1}+s_0)}\rho(\alpha)^{j+1}I(j)\\
&\le k \sum_{l=0}^{N-1} e^{-\alpha (b_{(l+1)k} +s_0)} \rho(\alpha)^{(l+1)k} I(lk).\qedhere
\end{align*}
\end{proof}

For the proof of Theorem~\ref{thm:percol}, let $(p_n)_{n \in \N}$ be a sequence of retention probabilities with $p_n \downarrow \pc$. For Theorem~\ref{thm:vary_f}, let $(t_n)_{n \in \N}$ be a sequence of parameters with $t_n \uparrow \infty$. We write,
\[
\rho_n(\cdot)= \begin{cases}
\rho^{p_n}(\cdot) = p_n \rho(\cdot) & \text{ for Theorem~\ref{thm:percol}}\\
\rho_{t_n}(\cdot) & \text{ for Theorem~\ref{thm:vary_f}}\end{cases} \quad \text{and} \quad 
\alpha_n^*:=\begin{cases}
\alpha^* & \text{ for Theorem~\ref{thm:percol}}\\
\alpha_{t_n}^* & \text{ for Theorem~\ref{thm:vary_f}.}\end{cases}
\]
Moreover, we denote by $\eigenfct_n$ the eigenfunction for $\rho_n(\alpha_n^*)$ from Lemma~\ref{lem:SpecProp}. 
The Markov chain from Lemma~\ref{lem:mto} corresponding to $\alpha=\alpha_n^*$ and retention parameter $p_n$ or attachment rule $f_{t_n}$ is denoted by $((S_i^\ssup{n},\type_i^\ssup{n})\colon i \in \N;\P)$, i.e.\ $\P=\P_{(s_0,\ell)}^{\alpha_n^*}$. One easily checks that in the setup of Theorem~\ref{thm:percol} the distribution of the Markov chain does not depend on $n$.

Finally, we introduce for all $n \in \N$,
\[
\eps_n:=\log \rho_n(\alpha_n^*) \quad \text{and} 
\quad \sigma^2:=\lim_{n \to \infty} \frac{\rho_n''(\alpha_n^*)}{\rho_n(\alpha_n^*)}.
\]
Notice, that in the situation of Theorem~\ref{thm:percol}, $\sigma^2=\rho''(\alpha^*)\pc$. The choice $\alpha=\alpha_n^*$ guarantees that for both theorems, $\eps_n\downarrow 0$ as $n \to \infty$.

\begin{lemma}\label{lem:upper_conv}
Let $N\in \N, s_0\leq 0$ and $\stepconst,\bdconst >0$. Let $k_n=\lfloor (\stepconst/\eps_n)^{\frac{3}{2}} N^{-1}\rfloor$ and $b_i^n(s_0)=a^{-1/2} \bdconst (Nk_n-i)^{1/3}-s_0$, $i \in [Nk_n]$. Then, for all $C>0$, $l \in [N]$,
\[
\limsup_{n \to \infty} \eps_n^{\frac{1}{2}} \log \Big( \sup_{- C/\sqrt{\eps_n} \leq s_0 \leq 0}  I_n^{s_0} (lk_n) \Big) 
 \le - \frac{\pi^2 \sigma^2}{2} \frac{l}{N} \frac{a^{3/2}}{(b + C )^2} . \]
where $I_n^s(j)=\P_{(s,\ell)}(-b_i^n(s) \le S_i^\ssup{n} \le 0\, \forall\, i\in [j])$.
\end{lemma}

Notice that the specific choice of parameters implies that
\begin{equation}\label{Ilk_eq}
I_n^{s_0}(lk_n)=\P_{(0,\ell)}^{\alpha_n^*}\Big(-a^{-1/2}\bdconst \Big(\frac{N}{l}-\frac{i}{lk_n}\Big)^{\frac{1}{3}} \le \frac{S_i^\ssup{n}}{(lk_n)^{1/3}} \le -\frac{s_0}{(lk_n)^{1/3}} \; \forall \,i\in [lk_n]\Big).
\end{equation}

\begin{proof}
Denote $g_1(t)=-a^{-1/2}\bdconst(\frac{N}{l}-t)^{\frac{1}{3}}$ and let $\delta >0$. For $n$ sufficiently large, \eqref{Ilk_eq} yields
\[
\sup_{- C /\sqrt{\eps_n} \leq s_0 \leq 0 }I_n^{s_0} (lk_n) \le \P_{(0,\ell)}^{\alpha_n^*}\Big(g_1\Big(\frac{i}{lk_n}\Big) \le \frac{S_i^\ssup{n}}{(lk_n)^{1/3}} \le \Big(\frac{l}{N}\Big)^{1/3} a^{-1/2} (C + \delta) \; \forall\, i\in [lk_n]\Big).
\]
By Lemma~\ref{lem:DonskCondCheck} we can apply Theorem~\ref{thm:mogulskii} with $a_n=(lk_n)^{1/3}$, $n=lk_n$ and $g_2(t)=a^{-1/2} (C+\delta) \frac{l}{N}$, and then take $\delta \searrow 0$ to derive
\begin{align*}
\limsup_{n \to \infty} \frac{1}{(lk_n)^{\frac{1}{3}}} \log \sup_{-C/\sqrt{\eps_n} \leq s_0 \leq 0 } I_n^{s_0}(lk_n) & \le - \frac{\pi^2 \sigma^2}{2} \int_0^1 \frac{1}{(a^{-1/2}\bdconst (\frac{N}{l}-t)^{\frac{1}{3}} + C (N/\ell)^{1/3}a^{-1/2})^2}\, dt\\
& \leq - \frac{\pi^2 \sigma^2}{2} \Big(\frac{l}{N}\Big)^{2/3} \frac{a}{(b + C )^2} . 
\end{align*}
Moreover, as $n \to \infty$,
\[
\eps_n^{\frac{1}{2}} (lk_n)^{\frac{1}{3}} = l^{\frac{1}{3}} \Big(\eps_n^{\frac{3}{2}} \lfloor (\stepconst/\eps_n)^{\frac{3}{2}} N^{-1}\rfloor\Big)^{\frac{1}{3}} \to l^{\frac{1}{3}} \frac{\stepconst^{\frac{1}{2}}}{N^{\frac{1}{3}}}.\qedhere
\]
\end{proof}

\begin{proof}[Proof of the upper bounds in Theorems~\ref{thm:percol} and \ref{thm:vary_f}] Since $\size(p,f)$ is non-decreasing in retention probability $p$ and attachment rule $f$, it suffices to consider the asymptotic behaviour along a discrete subsequence. 
As before, for Theorem~\ref{thm:percol} we take any discrete sequence of retention probabilities $p_n \downarrow \pc$ and for Theorem~\ref{thm:vary_f} we take any discrete sequence $t_n \uparrow \infty$. 
We make use of the notation introduced in and before Lemma~\ref{lem:upper_conv}. In particular, we fix $N \in \N, \stepconst>0$ and $\bdconst>0$ and let $k_n :=\lfloor (\stepconst/\eps_n)^{\frac{3}{2}} N^{-1}\rfloor$ and $b_i=b_i (s_0) = b_i^n=a^{-1/2}\bdconst(Nk_n-i)^{\frac{1}{3}}-s_0$. Write 
\[
\zeta^n_{s_0}=\begin{cases}
\zeta_{(s_0,\ell)}(p_n,f) &\text{for Theorem~\ref{thm:percol}},\\
\zeta_{(s_0,\ell)}(1,f_{t_n}) &\text{for Theorem~\ref{thm:vary_f}.}
\end{cases}
\]
 Applying Lemma~\ref{lem:upper_estim} with $\alpha=\alpha_n^*$, we obtain for every $N \in \N$, $\stepconst>0$, $\bdconst>0$, $C>0$
\begin{align*}
\log\Big( & \sup_{-C/\sqrt{\eps_n} \leq s_0 \leq 0} \zeta^n_{s_0}\Big)   - \log(N+1)\\&\le \max_{l\in [N-1]}\Big\{ \alpha_n^*  C \eps_n^{-1/2}+ k_nN \log \rho_n(\alpha_n^*) +\log \sup_{- C/\sqrt{\eps_n} \leq s_0 \leq 0}  I_n^{s_0}(k_nN),\\
&\hspace{2cm}\log k_n - \alpha_n^* (b_{k_n}^n+s_0)+ k_n \log \rho_n(\alpha_n^*) ,\\
&\hspace{2cm} \log k_n -\alpha_n^* (b_{(l+1)k_n}^n +s_0)+(l+1)k_n \log \rho_n(\alpha_n^*)+\log \sup_{- C/\sqrt{\eps_n} \leq s_0 \leq 0}   I_n^{s_0}(lk_n)\Big\}\\
&=\max_{l\in [N-1]}\Big\{\alpha_n^*  C \eps_n^{-1/2}+ k_nN \eps_n +\log\sup_{- C/\sqrt{\eps_n} \leq s_0 \leq 0}   I_n^{s_0}(k_nN) ,\\
&\hspace{2cm} \log k_n  - \alpha_n^* a^{-1/2} \bdconst (Nk_n-k_n)^{\frac{1}{3}} +k_n\eps_n,\\
&\hspace{2cm}  \log k_n -\alpha_n^* a^{-1/2}\bdconst (Nk_n - (l+1)k_n)^{\frac{1}{3}}+(l+1)k_n \eps_n +\log
\sup_{- C/\sqrt{\eps_n} \leq s_0 \leq 0} I_n^{s_0}(lk_n)\Big\}.
\end{align*}
Recall that the choice of parameters implies that
\[
\eps_n^{\frac{3}{2}}k_n \to \frac{\stepconst^{\frac{3}{2}}}{N}, \quad \text{and } \quad \sqrt{\eps_n} \log k_n  \to 0\qquad \text{as }n\to \infty.
\]
Hence, 
by Lemma~\ref{lem:upper_conv},
\begin{align*}
\limsup_{n \to \infty}\, \eps_n^{\frac{1}{2}}\log & \Big(\sup_{-C/\sqrt{\eps_n} \leq s_0 \leq 0} \zeta^n_{s_0} \Big)\\&\le \max_{l\in [N-1]}\Big\{ C \alpha^* + \stepconst^{\frac{3}{2}}-\stepconst^{\frac{3}{2}} \frac{ \pi^2\sigma^2}{2(b+C)^2},0 - \alpha^* \bdconst \frac{1}{N^{\frac{1}{3}}} (N-1)^{\frac{1}{3}} +\frac{\stepconst^{\frac{3}{2}}}{N},\\
& \phantom{\le \max_{l\in [N-1]}\Big\{}  0 -\alpha^* \bdconst \frac{1}{N^{\frac{1}{3}}} \big(N - (l+1)\big)^{\frac{1}{3}}+(l+1) \frac{\stepconst^{\frac{3}{2}}}{N} -\stepconst^{\frac{3}{2}} \frac{ \pi^2 \sigma^2}{2} \frac{1}{(b +C)^2} \Big( \frac{l}{N}\Big).
\end{align*}
Taking $N$ to infinity, we deduce 
\begin{align*}
\limsup_{n \to \infty} \eps_n^{\frac{1}{2}} &\log(\sup_{-C /\sqrt{\eps_n} \leq s \leq 0} \zeta^n_s) \le \max\Big\{C \alpha^* + \stepconst^{\frac{3}{2}}\big(1 -\frac{ \pi^2\sigma^2}{2} \frac{1}{(b + C)^2}\big) ,- \alpha^* \bdconst ,\\
&\phantom{\max\Big\{}\sup_{x \in (0,1]} \Big\{-\alpha^* \bdconst (1 - x)^{\frac{1}{3}}+\stepconst^{\frac{3}{2}} \Big( 1 -\frac{ \pi^2\sigma^2}{2} \frac{1}{(b + C)^2}\Big)  x\Big\} \\
&  = \max\Big\{C \alpha^* + \stepconst^{\frac{3}{2}}\big(1 -\frac{ \pi^2\sigma^2}{2} \frac{1}{(b + C)^2}\big) ,
\sup_{x \in [0,1]} \Big\{-\alpha^* \bdconst (1 - x)^{\frac{1}{3}}+\stepconst^{\frac{3}{2}}  \Big( 1 -\frac{ \pi^2\sigma^2}{2} \frac{1}{(b + C)^2}\Big) x \Big\}
\end{align*}
Now, we consider any $C < \sqrt{\frac{\pi^2 \sigma^2}{2}}$ and we choose 
$b$ such that
\[ b < \sqrt{\frac{\pi^2 \sigma^2}{2}} - C . \]
For these choices, we can take $a$ sufficiently large to see that the first term in the
above maximisation can be ignored, while in the second one the supremum 
is achieved at $x = 0$. This gives the bound 
\[ \limsup_{n \to \infty} \eps_n^{\frac{1}{2}} \log\Big(\sup_{-C /\sqrt{\eps_n} \leq s_0 \leq 0} \zeta^n_{s_0}\Big)
\leq - \alpha^* b. \]
Now, we can let $b \uparrow \sqrt{\frac{\pi^2 \sigma^2}{2}} - C$ to see that for any $0<C< \sqrt{\frac{\pi^2\sigma^2}{2}}$,
\begin{equation}\label{eq:0112-1} \limsup_{n \to \infty} \eps_n^{\frac{1}{2}} \log\Big(\sup_{-C /\sqrt{\eps_n} \leq s_0 \leq 0} \zeta^n_{s_0}\Big)
\leq - \alpha^* \Big(\sqrt{{\frac{\pi^2 \sigma^2}{2}}} - C\Big) . \end{equation}

To complete the proof take $\delta = \frac{1}{2} (1 - \alpha^*) > 0$. Choose
$\ell = \lfloor \alpha^* \frac{\pi \sigma}{\sqrt{2}\delta} \rfloor$
and note that  $\ell \delta \leq \alpha^* \frac{\pi \sigma}{\sqrt{2}} < \frac{\pi \sigma}{\sqrt{2}}$.
Then, we obtain from~\eqref{eq:0112-1} using the monotonicity of $s \mapsto \zeta_s^n$, 
\[ \begin{aligned}  \int_0^\infty e^{-s} \zeta_{-s}^n \, ds 
& \leq \sum_{k=0}^{\ell -1} \int_{k/\sqrt{\eps_n}}^{(k+1)/\sqrt{\eps_n}}
e^{-s} \zeta_s^n \, ds + e^{- \ell \delta  /\sqrt{\eps_n}} \\
& \leq \sum_{k=0}^{\ell-1}
e^{1/\sqrt{\eps_n} ( - \delta k - \alpha^* \frac{\pi \sigma}{\sqrt{2}}  + \alpha^* (k+1)\delta )(1+o(1))}  + e^{- ( \alpha^* \frac{\pi \sigma}{2}  \alpha^* - \delta) / \sqrt{\eps_n}} \\
& \leq e^{- 1/\sqrt{\eps_n} (\alpha^* \frac{\pi \sigma}{\sqrt{2}} - 2\delta) (1 + o(1))}  ,
\end{aligned} 
\]
where we used that $\alpha^* < 1$. Thus, 
we can deduce  by first taking the limit $n\ra\infty$ and then $\delta \downarrow 0$ that
\[ \limsup_{n \ra \infty} \sqrt{\eps_n} \log \int_0^\infty e^{-s} \zeta_{-s}^n \, ds  \leq - \alpha^* \sqrt{ \frac{\pi^2 \sigma^2}{2}}.\]
This immediately implies Theorem~\ref{thm:vary_f}  since $\sigma^2=\rho''(\alpha^*)$.
For Theorem~\ref{thm:percol}, $\pc=1/\rho(\alpha^*)$ implies
\[
\eps_n=\log(p_n\rho (\alpha^*))=\log(p_n/\pc)=\log\Big(1+\frac{p_n-\pc}{\pc}\Big)=\frac{p_n-\pc}{\pc}(1+o(1)).
\]
so that Theorem~\ref{thm:percol} follows since $\sigma^2=\rho''(\alpha^*)\pc$.
\end{proof}

\section{Proofs: Lower bound}\label{sec:lower_bound}

The general strategy of the lower bound is to identify a subtree of the IBRW that has the same distribution as a Galton-Watson tree. For this carefully chosen subtree we then lower bound the survival probability, which in return gives  the required lower bound on the survival probability of the IBRW. 
In Section~\ref{ssn:GW} we collect some general facts about Galton-Watson trees, which we will then use in Section~\ref{ssn:lower_bound} to carry out the proof of the lower bound.

\subsection{Galton-Watson lemmas}\label{ssn:GW}

In order to show the lower bound we construct a Galton-Watson tree, whose
particles are a subfamily of the killed branching random walk.
To estimate the survival probability of this Galton-Watson tree, we will use the 
following general lemma due to~\cite{GHS11} and we also recall the proof for completeness.

\begin{lemma}\label{le:GHS-GW}
Let ${\rm GW}$ be a Galton-Watson tree and denote by $X$ the number of children in the 
first generation and by $q$ its extinction probability. Then, for all $r \leq \min\{ \frac{1}{8}, q\}$, 
\begin{equation*} \p ( \rm {GW } \neq \emptyset ) \geq \p ( X \neq 0 ) - 2 r^{-2} \p ( 1 \leq X \leq r^{-2} ) - 2 r . \end{equation*}
\end{lemma}

\begin{proof} 
Denote by $q$ the extinction probability of ${\rm GW}$ and by $s\mapsto \momGen(s)=\E[s^{X}]$ the generating function of ${\rm GW}$.
Then, for every $r \in [0,q]$,
\begin{align}
q=\momGen(q)&= \momGen(0)+\int_0^{q} \momGen'(s) \, ds
\notag\\
&= \momGen(0)+\int_0^{q-r} \momGen'(s) \, ds +\int_{q-r}^{q} \momGen'(s) \, ds
\notag\\
&\le \P(X=0)+\momGen'(1-r) +r,\label{extinctProbGWbasicBd}
\end{align}
where we used that $\momGen'(s)$ is increasing, and $\momGen'(s)\le 1$ for all $s\in [0,\extinctProbGW]$.
We continue by estimating
\begin{align}
\momGen'(1-r)=\E\big[X (1-r)^{X-1}\big] &= \frac{1}{1-r} E\big[X(1-r)^{X}\big]
\notag\\
&\le  \frac{1}{1-r} \E\big[X e^{-r X}\big]
\notag \\
&= \frac{1}{1-r} \Big(\E\big[X e^{-r X}; X\le r^{-2}\big]+\E\big[X e^{-r X}; X> r^{-2}\big]\Big)
\notag\\
&\le  \frac{1}{1-r} \Big(r^{-2}\P(1\le X\le r^{-2})+r^{-2} e^{-1/r}\Big),\label{momGenDerivBd1}
\end{align}
where we used for the second summand that $u\mapsto u e^{-ru}$ is decreasing on $[r^{-1},\infty)$ and $[r^{-2}, \infty) \subseteq [r^{-1},\infty)$.
Then, for all $r \in (0,\frac{1}{8}]$,
\begin{equation}\label{eq:basicBounds}
 \frac{1}{1-r} \le 2 \qquad \text{and} \qquad \frac{1}{1-r} r^{-2} e^{-1/r} \le r. 
\end{equation}
Combining \eqref{momGenDerivBd1} and \eqref{eq:basicBounds}, we deduce
\begin{equation}\label{momGenDerivBd2}
\momGen'(1-r) \le  2r^{-2}\P(1\le X\le r^{-2})+r
\end{equation}
Rearranging \eqref{extinctProbGWbasicBd} and \eqref{momGenDerivBd2}, we conclude that for all $r \leq \min\{ \frac{1}{8}, q\}$,
\begin{equation}\label{StartingPoint}
1-q \ge \P(\offDistr \neq 0)- 2r^{-2} \P(1\le \offDistr\le r^{-2}) -2r, 
\end{equation}
as required.
\end{proof}

We will also need an estimate which guarantees that a supercritical Galton-Watson process grows exponentially fast on survival with large probability.

\begin{lemma}\label{le:exp_growth}
For all $\theta_1>1>\theta_2>0$ there exists $\delta>0$ such that for any Galton-Watson process 
$(X_n\colon n\in\N)$ with $X_0=1$ where mean offspring~$m$, offspring distribution $(p_k)_{k \in \N_0}$, and extinction probability $q$ satisfy 
$q, p_1<\delta$ and $m>1/\delta$, 
we have, for sufficiently large~$n$,
$$\P( X_n \leq \theta_1^n \, | \, X_n \geq 1 )\leq \theta_2^n.$$
\end{lemma}

\begin{proof}
Denote by $g$ the generating function of $(X_n\colon n\in\N)$. By pruning the tree, i.e.\ removing all finite subtrees, we obtain a tree which on survival of the original process equals a Galton-Watson process $(\tilde X_n\colon n\in\N)$ with
$\tilde p_0=0$, $\tilde p_1=f'(q) \leq p_1+ \frac{2q}{(1-q)^2}$, and the same mean as the original process. 
Hence $\delta>0$ can be chosen such that the pruned process has arbitrarily small $\tilde p_1$ and arbitrarily large mean.
\smallskip

Under this assumption we now calculate
\begin{align*}
\P(\tilde X_n \leq \theta_1^n ) & 
\leq  \sum_{j=1}^{\lceil n/2 \rceil} \P\big( \tilde X_n \leq \theta_1^n, \mbox{ first branching in generation } j \big) \\
& \phantom{jodldodldidu} +  \P\big( \mbox{ no branching up to generation }\lceil n/2 \rceil\big) \\
& \leq 
\sum_{j=1}^{\lceil n/2 \rceil}  \big(\P( \tilde X_{n-j} \leq \theta_1^n)\big)^2 \tilde p_1^{j-1} + \tilde p_1^{n/2}\\
& \leq \tilde p_1^{-1} \P(\tilde X_n \leq \theta_1^n ) \P(\tilde X_{\lceil n/2 \rceil} \leq \theta_1^n ) + \tilde p_1^{n/2}.\\
\end{align*}
Hence
$$\P(\tilde X_n \leq \theta_1^n )
\leq \frac{\tilde p_1^{n/2}}{1-\tilde p_1^{-1}\P(\tilde X_{\lceil n/2 \rceil} \leq \theta_1^n )}.$$
Choosing $\delta>0$ so small that $\E \tilde X_1>\theta_1^2$ and $\sqrt{\tilde p_1}<\theta_2$, 
for sufficiently large~$n$, the denominator is larger than $1/2$ and the numerator is smaller than $\frac{1}{8} \theta_2^n$, so that
\[ \P(\tilde X_n \leq \theta_1^n ) \leq \frac{1}{4} \theta_2^n. \]

Note that if we condition $X_n$ on extinction in distribution it is equal to a Galton-Watson process $X_n^*$ with mean $f'(q) = \tilde p_1$. Therefore, by Markov's inequality
\[ \p (X_n \geq 1 \, |\, \mbox{ extinction } ) = \p ( X_n^* \geq 1 ) \leq \E[ X_n^*] = \tilde p_1^n \leq \frac{1}{4} \theta_2^n, \]
by the same assumptions on $\delta$ as above for $n$ large.
We can also assume that $\delta$ is sufficiently small, so that the extinction probability $q$ is less than $1/2$. 
Hence, combining the above estimates, 
\[ \begin{aligned} \p ( X_n \leq \theta_1 \, |\, X_n \geq 1) 
& \leq \frac{1}{\p (X_n \geq 1)} \Big\{  \p (X_n \leq \theta_1 \, |\, \mbox{survival}\,) + \p (X_n \geq 1  \, | \, \mbox{extinction}\,) \Big\} \\
& \leq \frac{1}{1 - q} \big( \frac{1}{4} \theta_2^n + \frac{1}{4} \theta_2^n \big) \leq \theta_2^n ,  
\end{aligned} \]
which completes the proof.
\end{proof}

\subsection{The lower bound}\label{ssn:lower_bound}

Throughout, we use  the same notation as in the upper bound: 
For the proof of Theorem~\ref{thm:percol}, let $(p_n)_{n \in \N}$ be a sequence of retention probabilities with $p_n \downarrow \pc$. For Theorem~\ref{thm:vary_f}, let $(t_n)_{n \in \N}$ be a sequence of parameters with $t_n \uparrow \infty$. We denote by $S^\ssup{n}$ the positions either in the percolated IBRW or in the IBRW with attachment rule $f_{t_n}$. If the context is clear, we will omit the superscript. Also, we write,
\[
\rho_n(\cdot)= \begin{cases}
p_n \rho(\cdot) & \text{ for Theorem~\ref{thm:percol}}\\
\rho_{t_n}(\cdot) & \text{ for Theorem~\ref{thm:vary_f}}\end{cases} \quad \text{and} \quad 
\alpha_n^*:=\begin{cases}
\alpha^* & \text{ for Theorem~\ref{thm:percol}}\\
\alpha_{t_n}^* & \text{ for Theorem~\ref{thm:vary_f}.}\end{cases}
\]
Moreover, we denote by $\eigenfct_n$ the eigenfunction for $\rho_n(\alpha_n^*)$ from Lemma~\ref{lem:SpecProp}. 

Finally, we introduce for all $n \in \N$,
\begin{equation}\label{eq:sigma}
\eps_n:=\log \rho_n(\alpha_n^*) \quad \text{and} 
\quad \sigma^2:=\lim_{n \to \infty} \frac{\rho_n''(\alpha_n^*)}{\rho_n(\alpha_n^*)}.
\end{equation}
Notice, that in the situation of Theorem~\ref{thm:percol}, $\sigma^2=\rho''(\alpha^*)\pc$. The choice $\alpha=\alpha_n^*$ guarantees by Proposition~\ref{prop:conv} that for both theorems, $\eps_n\downarrow 0$ as $n \to \infty$.

Given any starting point $s\geq 0$ and initial type $\tau$, we will write 
$P = P_{(s,\tau)} =  P^{p_n}_{(s,\tau)}$ in the percolation case and $P = P_{(s,\tau)} = P^{1}_{(s,\tau)}$ in the case of Theorem~\ref{thm:vary_f}.

In view of Lemma~\ref{le:GHS-GW}, we will now choose a 
Galton-Watson tree $\subGW$ as a subtree of the killed IBRW in the following way, where we denote by $X^\ssup{n}$ the number of children in the first generation of $\subGW$:
\begin{enumerate}
\item[(a)] $P(X^\ssup{n} \neq 0) \approx $ the survival probability of the Galton-Watson process. That is, when there are offspring, then the process usually survives.
\item[(b)] $P(X^\ssup{n} \neq 0)$ is close to the survival probability of the killed IBRW. That means that we chose the subpopulation as a good approximation of the BRW and that the first inequality in \eqref{StartingPoint} is a good estimate.
\item[(c)] $P(1\le X^\ssup{n} \le r^{-2})$ has to be small to beat $r^{2}$. 
\end{enumerate}

The Galton-Watson tree is obtained by a coarse-graining procedure, which we now describe. It involves positive parameters $b, \lambda, \theta, M$ which will be chosen carefully at a later stage of the proof. We group together the first $N+o(N)$ generations in the IBRW to form the first generation in $\subGW$.
It turns that we have to choose $N = N_n$ such that
\[ N_n = \lfloor (b/\eps_n)^{3/2} \rfloor  , \]
and for the first $N$ steps we only choose particles whose positions are in the interval
\[
I_{i,n}= [-\tfrac{\theta b}{N_n^{2/3}} i - \lambda N_n^{1/3}, -\tfrac{\theta b}{N_n^{2/3}} i].
\]

To be precise, let $L(N) = N + \lfloor N^{1/3} \rfloor$ and 
\begin{align*}
\subOff&=\{ x \colon \abs{x}=L(N_n), \loc(x_i)-\loc(x_0)  \in I_{i,n}, 
 \text{ for }i=1,\ldots, N_n,\\
 & \hspace{3cm} \loc(x_i) - \loc(x_{i-1}) \leq M, \text{ for } i=N_n+1, \ldots, L(N_n) \}. 
\end{align*}
Then we define $\subOff$ to be the particles in the first generation of $\subGW$. We include the last $\lfloor N^{1/3} \rfloor$ generations to make sure that if we survive until time $N$, then we will have many particles by time $L(N)$.

We iterate the procedure, i.e.\ the children of $y \in \cC_n$ will be
\[ \begin{aligned}
\subOff(y)&=\{ x \colon \abs{x}=2L(N_n), y \leq x , \loc(x_i)-\loc(y)  \in I_{i,n}, 
 \text{ for }i=L(N_n)+ 1,\ldots, L(N_n) + N_n,\\
 & \hspace{3cm} \loc(x_i) - \loc(x_{i-1}) \leq M, \text{ for } i= L(N_n) + N_n+1, \ldots,2 L(N_n) \}. \end{aligned}\]
 and $\bigcup_{y \in \subOff} \subOff(y)$ will form the individuals of the second generation of $\subGW$ and we continue in a similar way.
 Note that by the construction of $I_{i,n}$, we only include children in the second generation of the original IBRW that are to the left of the position of their ancestor. In particular, their distribution does not depend on the type of the parent. Moreover, all the conditions on the spatial positions are relative to $S(\emptyset)$. 
 Therefore, the distribution of $\subOff$ does not depend on either  the type of the root $\emptyset$ nor the initial position $S(\emptyset)$.
 Similarly, the distribution of $\subOff(y)$ does not depend on either type nor position of $y$. 
Hence, the number of individuals in the different generations really do form a Galton-Watson process.

Moreover, if we assume that
\begin{equation}\label{eq:3110-1} M < \theta b, \end{equation}
then  we have that 
all particle positions satisfy $\loc(x_i) -\loc(x_0)\leq 0 $ for all $i \leq |x|$ and $ x\in \subOff$
and $\subGW$ is really a subset of the killed branching random walk.

{\bf Coupling with a Galton-Watson process}

To control the contribution of the last $N^{1/3}$ steps of the branching random walk, 
we use the following coupling: 
We can couple the IBRW with a modified IBRW, where in generations 
$k L(N) + N,\ldots,k L(N) +  L(N)-1$, for $k \in \N_0$,  the particles place their offspring according to the following rules relative to their own position:
\begin{itemize}	
	\item[(i)] to the right, the positions of the offspring follow the jump times of the birth process $Z_t^\ssup{f}$ started in $0$, which jumps from $k$ to $k+1$ at rate $f(k)$. 
	\item[(ii)] to the left, the positions of the offspring are given by a Poisson point process with intensity
\[ e^t E[f(Z_{-t}^\ssup{f})] \1_{(-\infty,0]}(t) \, dt ,  \]
where $(Z_t^\ssup{f})$ is the birth process with jump rate $f$ started in $0$.	
	\item[(iii)] Also, in these generations types do not play a role, so we define all particles to have the same type as $S(\emptyset)$ in the original process.
\end{itemize}
Now, we define the Galton-Watson $\underline {\rm GW}_n$ similarly to above in terms of the modified IBRW, where we 
again denote by $\underline \cC_n$ the individuals in the first generation. Since $(f_{t_k})$ is decreasing, 
we can couple the processes such that if $\underline {\rm GW}_n$ survives then $\subGW$ survives and also
such that $|\underline \cC_n| \leq |\cC_n|$.

For the lower bound on $P ( \subOff \neq \emptyset)$ it turns out that it is enough to control 
the probability of the set
\[ \OffAsymp =  \{ x \colon \abs{x}=N_n, \loc(x_i)-\loc(x_0)  \in I_{i,n} , \mbox{ for } i = 1, \ldots, N_n \} . \]
being non-empty. 
Note that for both, the original process and the modification, the set $\OffAsymp$ is the same, however in 
the modified process, the next generations $N+1, \ldots, L(N)$ have the distribution of a single-type 
branching random walk. In particular, the number of particles in each generation form a 
Galton-Watson process, which we will denote by $(\underline X_k)_{k \in \N_0} = (\underline X_k(M))_{k \in \N_0}$. Moreover,  we will denote by $\underline q = \underline q_M$ its extinction probability when started with a single particle.
Note that $\underline q$ does not depend on $n$  and  we will use
that by increasing $M$  we have that $\lim_{M \ra \infty} \underline q_M = 0$ and 
also  $\lim_{M \ra \infty} \E \underline X_1(M)  = \infty$.

By the Markov property, the survival of the two subsets of the killed branching branching walk 
are related by 
\begin{equation}\label{eq:asymp_ok} P ( \subOff \neq \emptyset) \geq P( \underline \cC_n \neq \emptyset) \geq (1-  \underline q_M) \p ( \OffAsymp \neq \emptyset) . \end{equation}

The next result is
the key step in the overall lower bound, where we will bound the probability that $\OffAsymp$ is non-empty.

\begin{proposition}\label{le:thining_survival} For any $\alpha^* \theta < 1$ and any $\la^2 > \frac{\pi^2 \sigma^2}{2 b (1- \alpha^*\theta)}$, we have
for any initial position $s \geq 0$ and initial type $\tau \in\typespace$,
\[ \liminf_{n \ra \infty} N_n^{-1/3} \log P_{(s,\tau)} \big( \OffAsymp \neq \emptyset \big)\geq - \la \alpha^* . \]
\end{proposition}

 \begin{proof}
By construction the probability of the event $\OffAsymp \neq \emptyset$  does not depend on the starting point of the initial point nor its type, so we can assume that $S^\ssup{n}(\emptyset) = 0$ and $\tau(\emptyset) = 0$.

For the first part of the proof, we will omit the indices $n$, whenever the context is clear and we are not dealing with asymptotic statements.
In particular, we will write $N = N_n, S = S^\ssup{n}$, $\alpha = \alpha_n$, etc. Also, in this proof $\rho = \rho^{p_n}$ in the percolation case and $\rho = \rho_{t_n}$ if the attachment rule is varying.

The first step is to carefully, select the relevant particles in $\OffAsymp$. 
For  $R_n = e^{N^{1/4}}$
and $M_n = N^{1/5}$ and an individual $x$, we write 
\[ \Delta S(x_i) = S(x_i) - S(x_{i-1}), \quad i \leq |x| . \]
Recall that
\[ \nu_x = \sum_{y \, : \, y- =x } \delta_{\Delta S(y)} . \]
For any $|x| = N$, let 
\[ \bar \nu_{x_{i-1}} := \nu_{x_{i-1}} \big( ( - \infty, M_n) \big). \]
Define
\[
Y_n= \#\{ x \, : \, \abs{x}=N , \loc(x_i)\in I_{i,n},  \bar\nu_{x_{i-1}} \leq R_n, \Delta S(x_i) \leq M_n\, \forall i \in [N]\}.
\]
The Paley-Zygmund inequality yields 
\begin{equation}
P(\OffAsymp \neq \emptyset ) \geq  P(Y_n \ge 1) \ge \frac{E[Y_n]^2}{E[Y_n^2]}.
\end{equation}
The remaining proof proceeds as follows: in the first step we find an easier upper bound on $\E[Y_n^2]$, which we can estimate using the many-to-one lemma in the second step. In Step 3, we find a lower bound on $\E[Y_n]$, which we will then combine with the other steps to obtain our claim.

\emph{Step 1: Upper bound on $\E[Y_n^2]$}. First write
 \begin{equation*}
 \begin{aligned}
E[Y_n^2] & = E\Big[ \sum_{\abs{x}=N} \sum_{\abs{y}=N}  \mathbbm{1}{\{\loc(x_i) \in I_{i,n}, \bar \nu_{x_{i-1}} \leq R_n, \Delta S(x_i) \leq M_n\ \forall i \leq N\}} \\[-3mm]
& \hspace{4cm} \times \1{\{\loc(y_i) \in I_{i,n} , \bar \nu_{y_{i-1}} \leq R_n, \Delta S(y_i) \leq M_n \, \forall i \le N\}}\Big]
\end{aligned}
\end{equation*}
We split the sum over $y$ according to the last time $j \in \{ 0, \ldots, N\}$
that the ancestors of $y$ agree with $x$ to obtain 
 \[\begin{aligned}
E[Y_n^2] & = 
E\Big[ \sum_{|x| = N } \1\{\loc(x_i)\in I_{i,n},  \bar\nu_{x_{i-1}} \leq R_n, \Delta S(x_i) \leq M_n\, \forall i \in [N]\} \\
& \quad \times \sum_{j=0}^N 
\sum_{|y| = N} \1\{ y_i = x_i \forall i \leq j, y_{j+1} \neq x_{j+1} ,  \loc(y_i)\in I_{i,n},  \bar\nu_{y_{i-1}} \leq R_n, \Delta S(y_i) \leq M_n\, \forall i \in [N] \}\Big]. 
\end{aligned} \]
Conditioning on the $j+1$ first generations and using the independence of the branching process and the fact that on $\{ \bar \nu_{x_j} \leq R_n\}$, 
we only have to consider at most $R_n$ relevant siblings of $x_{j+1}$, 
we obtain the upper bound
\[
 E[Y_n^2] \leq 
E\Big[ \sum_{|x| = N } \1\{\loc(x_i)\in I_{i,n},  \bar\nu_{x_{i-1}} \leq R_n, \Delta S(x_i) \leq M_n\, \forall i \in [N]\} \Big]
\Big(1+ \sum_{j=0}^{N-1}R_n h_{j+1,n}\Big), \]
where $h_{N,n} = 1$ and for $j \leq N-2$,
\[ h_{j,n} := \sup_{s_0 \in I_{j,n} , \type_0 \in \typespace } E_{(s_0, \type_0)}\Big[\sum_{\abs{y'}=N-j} \mathbbm{1}\set{\loc(y_i') \in I_{i+j,n} \,\forall  i\le N-j)}\Big] . \]

In particular, we have shown that 
\begin{equation*}
E[Y_n^2] \le E[Y_n] \Big(1+ \sum_{j=0}^{n-1}h_{j+1,n} R_n \Big), 
\end{equation*}
leading to 
\begin{equation}\label{eq:second_moment}
P(\OffAsymp \neq \emptyset ) \ge \frac{E[Y_n]}{1+\sum_{j=0}^{n-1}h_{j+1,n} R_n}.
\end{equation}

\emph{Step 2: Upper bound on $h_{j,n}$.}
First, we note since the left, resp.\ right, end point of $I_{k,n}$ are to the left of the left, resp. right, end point of 
$I_{j,n}$ for $k \geq j$, we can apply the monotonicity in the initial types to deduce that
\[ h_{j,n} \leq \sup_{s_0 \in I_{j,n}  } E_{(0, 0)}\Big[\sum_{\abs{y'}=N-j} \mathbbm{1}\set{\loc(y_i')+s_0 \in I_{i+j,n}, \,\forall  i\le N-j}\Big]  \]

Thus, we obtain by the many-to-one Lemma~\ref{lem:mto}
\[\begin{aligned} h_{j,n} &\leq \sup_{s_0 \in I_{j,n}  } E_{(0, 0)}\Big[\sum_{\abs{y'}=N-j} \mathbbm{1}\set{\loc(y_i')+s_0 \in I_{i+j,n} \,\forall  i\le N-j}\Big]  \\
&  = \sup_{s_0 \in I_{j,n}}
\E^\alpha_{(0,0)} \Big[e^{\alpha S_{N-j}} \frac{v_\alpha(\tau_0)}{v_\alpha(\tau_N)} \rho(\alpha)^{N-j}\1\{ S_i + s_0 \in I_{i+j,n} \,  \forall i \leq N-j\} \Big]
\end{aligned}\]
By the monotonicity in types, Lemma~\ref{lem:SpecProp}, we have that $v_\alpha(\tau_0) / v_\alpha(\tau_N) \leq v_\alpha(0) / v_\alpha(\ell)$ and further by Proposition~\ref{prop:conv}, $C := \sup_{n} \frac{v^\ssup{n}_{\alpha_n}(0)}{v^\ssup{n}_{\alpha_n}(\ell)}< \infty$, so that
\begin{equation}\label{eq:2910-1}\begin{aligned}
h_{j,n} & \leq C \sup_{s_0 \in I_{j,n}}
\E^\alpha_{(0,0)} \Big[e^{\alpha S_{N-j} } \rho(\alpha)^{N-j}\1\{ S_i + s_0 \in I_{i+j,n}  \,\forall i \leq N-j\} \Big] \\
& \leq C \sup_{u \in [-\la N^{1/3}, 0]}
e^{- \alpha \theta \eps_n (N - j) + \la \alpha N^{1/3}  }\rho(\alpha)^{N-j}\,
\p^\alpha_{(0,0)} \big( S_i + u  -\tfrac{\theta b}{N^{2/3}} j \in I_{i+j,n}  \,\forall i \leq N-j\} \big) , 
\end{aligned}\end{equation}

We now would like to apply Mogulskii's theorem to estimate the probability, but
 we need a bound that works uniformly for all $j$ (since we will be summing over $j$) and uniformly in $u$. 
We thus approximate the sum over $j$ by a finite sum and also split up the interval $[-\la N^{1/3},0]$ into smaller subintervals. 
So fix $\kappa \in \N$ and define 
\[ K := K_n := \lfloor N/\kappa \rfloor. \]
For the next estimate, suppose that $j \in [ (j' -1) K , j' K -  1]$ for some $j' \in \{1, \ldots, \kappa -1 \}$ and
assume $u \in [- q \la N^{1/3}/ \kappa, - (q-1) \la N^{1/3}/\kappa]$ for some $q \in \{1, \ldots, \kappa\}$. Then,
\[\begin{aligned}   \big\{ S_i & + u   - \tfrac{\theta b}{N^{2/3}} j \in I_{i+j,n}  \,\forall i \leq N-j\} \big\} \\
& = \big\{ - \tfrac{\theta b}{N^{2/3}}  (i+j) - \la N^{1/3} \leq S_i + u  - \tfrac{\theta b}{N^{2/3}} j \leq - \tfrac{\theta b}{N^{2/3}} (i+j)   \,\forall i \leq N-j \big\} \\
& \subseteq  \big\{  - \tfrac{\theta b}{N^{2/3}}  i - \la N^{1/3} \leq S_i + u  \leq - \tfrac{\theta b}{N^{2/3}}  i   \,\forall i \leq (\kappa - j') K \big\} \\ 
& \subseteq \Big\{ (q-1)\la N^{1/3}/ \kappa  -  \tfrac{\theta b}{N^{2/3}} i - \la N^{1/3} \leq S_i   \leq -  \tfrac{\theta b}{N^{2/3}} i + q \la N^{1/3}/\kappa  \,\forall i \leq (\kappa - j') K \Big\} \\ 
& \subseteq \Big\{ (q-1)\la / \kappa  - \tfrac{ \theta b}{N} i - \la  \leq \tfrac{S_i}{N^{1/3}}   \leq -  \tfrac{\theta b}{N} i + q \la/\kappa  \,\forall i \leq (\kappa - j') K \Big\} \\ 
& \subseteq \Big\{ \frac{ (q-1 - \kappa)\la}{\kappa} - \theta b \frac{\kappa - j'}{\kappa} \frac{i}{(\kappa - j') K}   \leq \frac{S_i }{N^{1/3}}   \leq - \theta b \frac{\kappa - j' - 1}{\kappa } \frac{i }{(\kappa - j') K}  + \frac{q \la}{\kappa}  \,\forall i \leq\ (\kappa - j') K \Big\} \\
& =: E_{j',n} ,
\end{aligned} 
\]

where we assumed that $n$ is sufficiently large so that
\[  \frac{i}{N} \leq \frac{\kappa - j'}{\kappa}  \frac{i }{(\kappa - j') K} \quad\mbox{and} \quad
 \frac{i}{N} \geq \frac{\kappa - j'}{\kappa}  \frac{i }{(\kappa - j') K} \frac{K}{K+1} 
\geq  \frac{\kappa - j'}{\kappa} \Big(1 - \frac{1}{\kappa - j'}\Big)\frac{i }{(\kappa - j') K} . \]

By Lemma~\ref{lem:DonskCondCheck} we can apply
 Mogulskii's theorem, Theorem~\ref{thm:mogulskii} with 
$g_1(t) = \frac{(q-1 - \kappa)\la}{\kappa} - \theta b \frac{\kappa - j'}{\kappa} t$ and $g_2(t) = 
- \theta b \frac{\kappa - j' - 1}{\kappa } t + \frac{q \la}{\kappa}$.
\[ \lim_{n \ra \infty} \frac{N_n^{2/3}}{(\kappa - j')K_n} \log \p^{\alpha_n}_{(0,0)} (E_{j',n}) = - \frac{\pi^2 \sigma^2}{2}
\int_0^1 \frac{1}{(g_2(t) - g_1(t))^2} \, d t =: -C(\kappa), \]
where $\sigma$ is defined in~\eqref{eq:sigma} and we noticed that the integral does not depend on $q$ nor $j'$.

Hence, applying the display~\eqref{eq:2910-1} , where we use the definition of $\eps_n= \rho(\alpha)$, 
and also that $\alpha_n \theta < 1$ (since $\alpha^*\theta < 1$, for all $n$ sufficiently large, uniformly in $j \in [(j'-1)K, j'K-1]$, we find that
\[ \begin{aligned} \log h_{j,n} & \leq (- \alpha_n \theta + 1 ) \eps_n (N_n-j) + \la\alpha_n N_n^{1/3} - \frac{(\kappa - j') K_n}{N_n^{2/3}} C(\kappa) (1+o(1)) \\
& \leq (- \alpha_n \theta + 1 ) \frac{b}{(N_n+1)^{2/3}}  (N_n-j) + \alpha_n\la N_n^{1/3} - \frac{(\kappa - j') K_n}{N_n^{2/3}} C(\kappa) (1+o(1)) \\
& \leq N^{1/3} ( - \alpha^* \theta b + b + \alpha^*\la - C(\kappa) ) (1+o(1)) +
j' K_n (  - b( 1 - \alpha^* \theta ) + C(\kappa) ))(1+o(1))  . 
\end{aligned} \]
Note that $C(\kappa) \ra  \frac{\pi^2\sigma^2}{2\la^2}$ as $\kappa\ra\infty$. However, by assumption we have that $\frac{\pi^2\sigma^2}{2\la^2} < b(1-\alpha^* \theta)$, so we can assume that $\kappa$ is large enough such that $C(\kappa) < b (1 - \alpha^* \theta)$. In this case, we can deduce that 
\begin{equation}\label{eq:2509-1}  \log h_{j,n} \leq N^{1/3} ( - \alpha^* \theta b + b + \alpha^*\la - C(\kappa) ) (1+o(1)) . \end{equation}
Moreover, if $j \in [(\kappa - 1)K+1, N]$, then we can use the trivial bound 
\begin{equation}\label{eq:2509-2}
\log (h_{j,n} / C) \leq   \alpha_n \la N^{1/3} + (N-j) \eps_n 
\leq \alpha_n\la N^{1/3} +  \frac{b}{N^{2/3}} 2 K 
\leq N^{1/3} (\alpha^* \la + 2b /\kappa)(1+o(1)). 
\end{equation}
Consequently,  if we combine~\eqref{eq:2509-1} and~\eqref{eq:2509-2} we obtain 
\[ \limsup_{n \ra \infty} N_n^{-1/3} \log \sum_{j=1}^N h_{j,n}
\leq b(1 - \alpha^* \theta + 2/\kappa)+ \alpha^* \la - C(\kappa)  . \]
Finally, letting $\kappa \ra \infty$ gives
\[ \limsup_{n \ra \infty} N_n^{-1/3} \log \sum_{j=1}^N h_{j,n}
\leq  - \alpha^* \theta b + b + \alpha^* \la - \frac{\pi^2 \sigma^2}{2\la^2}. \]

\emph{Step 3:  Lower bounding $E [ Y_n]$.}
Let $\theta_1 > \theta$, then if we set $\Delta S_i = S_i - S_{i-1}$ and $\bar \nu_i = \nu_i((-\infty,M_n))$, we have by Lemma~\ref{le:mto_2} that
\[ \begin{aligned} E[Y_n] &  = \E^\alpha_{(0,0)} \Big[ e^{\alpha \loc_N} \rho(\alpha)^N  \frac{v_\alpha(0)}{v_\alpha(\tau_N)}\1_{\{ S_i\in I_{i,n}, \bar\nu_{i-1} \leq R_n, \Delta S_i \leq M_n \forall i \leq N \}} \Big] \\
& \geq \E^\alpha_{(0,0)} [ e^{\alpha \loc_N} \rho(\alpha)^N \1_{\{ S_i\in I_{i,n}, \bar\nu_{i-1} \leq R_n, \Delta S_i \leq M_n \forall i \leq N, S_N \geq - \theta_1 b N^{1/3} \}}  ] \\
& \geq e^{- \alpha \theta_1 b N^{1/3}} \rho(\alpha)^N \p^\alpha_{(0,0)} \big( S_i\in I_{i,n}, \bar\nu_{i-1} \leq R_n, \Delta S_i \leq M_n \forall i \leq N, S_N \geq - \theta_1 b N^{1/3} \big), 
   \end{aligned}
\]
where we also used that by the monotonicity of types, see Lemma~\ref{lem:SpecProp}, $v_\alpha(0) \geq v_\alpha(s)$ for any $s \in \typespace$.

Define the Markov chain 
$(\tilde S_k, \tilde \tau_k, \tilde \nu_{k-1})_{k \in \N}$ 
with associated filtration $( \mathcal{F}_k)$ and transitions given for any measurable $F$ by 
\[ \E^\alpha [ F(\tilde S_k - \tilde S_{k-1}, \tilde \tau_{\tau_k} , \tilde \nu_{k-1} )\, |\, \mathcal{F}_{k-1} ] 
= \E^{\alpha}_{(0,\tau_{k-1})} [ F(S_1, \tau_1, \nu_0) \, |\, \nu_0((-\infty, M_n) ) \leq R_n, S_1\leq M_n ] .
\]
Then, we can continue the previous display as
\begin{equation}\label{eq:2710-1} \begin{aligned} E[Y_n] & \geq e^{- \alpha \theta_1 b N^{1/3}} \rho(\alpha)^N \p^\alpha_{(0,0)} \big( \tilde S_i\in I_{i,n} \forall i  \leq N, \tilde  S_N \geq - \theta_1 b N^{1/3} \big)\\ 
&\hspace{4cm} \times \inf_{\tau \in \typespace} \p^\alpha_{(0,\tau)}( \bar\nu_{0} \leq R_n, \Delta S_1 \leq M_n )^N . 
\end{aligned}\end{equation}

Note that
\[\begin{aligned}
 \p^\alpha_{(0,0)}  \big( \tilde S_i\in I_{i,n} \forall i \in \leq N, \tilde S_N \geq - \theta_1 b N^{1/3} \big)
& = \p^\alpha_{(0,0)}  \Big(- \frac{\theta b}{N} i- \la  \leq \frac{\tilde S_i}{N^{1/3}} \leq - \frac{\theta b}{N} i \,  \forall i  \leq N,\tilde  S_N \geq - \theta_1 b N^{1/3} \Big)
\end{aligned} \]

By Lemma~\ref{lem:DonskCondCheck2} we can apply 
 Mogulskii's Theorem~\ref{thm:mogulskii} also to $(\tilde S_k)_{k \in \N_0}$
 with $g_1(t) = -\theta b t - \la$, $g_2(t) = -\theta b t$, so that 
we get 
\begin{equation}\label{eq:2710-2} \lim_{n \ra \infty} N_n^{-1/3} \log \p^{\alpha_n}_{(0,0)}  \Big(- \frac{\theta b}{N} i- \la  \leq \frac{\tilde S^\ssup{n}_i}{N^{1/3}} \leq - \frac{\theta b}{N} i \  \forall i  \leq N, \tilde S_N \geq - \theta_1 b N^{1/3} \Big) =  - \frac{\pi^2 \sigma^2}{2 \la^2} .\end{equation}

We estimate the last term in~\eqref{eq:2710-1} by deducing from Corollary~\ref{cor:negligible} that there exist
$\tilde C, \tilde \gamma > 0$ such that
\[ \inf_{\tau \in \typespace} \p^\alpha_{(0,\tau)}( \bar\nu_{0} \leq R_1, \Delta S_1 \leq M_n )^N
\geq (1 - \tilde C e^{-\tilde \gamma N^{1/5}}  )^N , \]
so that in particular, 
\begin{equation}\label{eq:2710-3} \lim_{N \ra \infty} N^{-1/3} \log \Big( \inf_{\tau \in \typespace} \p^\alpha_{(0,\tau)}( \bar\nu_{0} \leq R_n, \Delta S_1 \leq M_n )^N\Big)  = 0 . \end{equation}
Combining~\eqref{eq:2710-2} and~\eqref{eq:2710-3}, we get from~\eqref{eq:2710-1} 
if we recall the definition of $\eps_n$ that
\[ \liminf_{n \ra \infty} N_n^{-1/3} \log E[Y_n] 
\geq - \alpha \theta_1 b + b - \frac{\pi^2 \sigma^2}{2 \la^2} . \]

\emph{Step 4: Combining the estimates.}

Combining the upper bound on $\sum_{j=1}^N h_{j,n}$ with the lower bound on $E[Y_n]$
as well as the fact that $R_n = e^{N^{1/4}}$, we obtain from the second moment bound~\eqref{eq:second_moment}, 
\[ \liminf_{n \ra \infty} N_n^{-1/3} \log P \big( \OffAsymp\neq \emptyset \big)  \geq - \alpha \theta_1 b + b - \frac{\pi^2 \sigma^2}{2 \la^2} + \alpha b \theta - b -\alpha \la + \frac{\pi^2\sigma^2}{2 \la^2} = 
- \alpha\la
+ \alpha b(\theta - \theta_1) .  \]
Finally, we can let $\theta_1 \downarrow \theta$ 
to obtain the claim of the proposition.
\end{proof}

\begin{proof}[Proof of Theorems~\ref{thm:percol} and~\ref{thm:vary_f} -- lower bounds]
By the same argument as at the end of the proof of the upper bound, we have completed the proofs of Theorems~\ref{thm:percol} and~\ref{thm:vary_f}
if we can show that if $\zeta_n$ is either $\theta(p_n,f)$ or $\theta(1,f_{t_n})$, then 
\begin{equation}\label{eq:compl} \lim_{ n \ra \infty} \sqrt{\eps_n} \log \zeta_n \geq - \sqrt{\frac{\pi^2 \sigma^2}{2}} \alpha^* . \end{equation}

In our above construction, we first of all choose the constant $M$ large enough such that
the Galton-Watson process $(\underline X_k(M))_{k \in \N_0}$ satisfies the assumptions of Lemma~\ref{le:exp_growth}
for $\theta_1 = 2$ and $\theta_2 = 1/2$. Also, we can assume that its survival probability satisfies $1- \underline q_M \geq 1/2$.
Let $\theta$ be such that $\alpha^* \theta < 1$. 
Then, choose $b$  large enough such that $\theta b > M$, so that in particular $\subGW$ is a subset of the killed IBRW, cf.~\eqref{eq:3110-1}. Additionally, we require
that 
\begin{equation}\label{eq:b_large} b > \frac{\pi^2 \sigma^2}{(1-\alpha^*\theta)}  \Big(\frac{2 \alpha^*}{\log 2}\Big)^2  
 . \end{equation}

By Proposition~\ref{le:thining_survival} 
 we obtain that for any $\lambda > \sqrt{ \frac{\pi^2 \sigma^2}{2 b (1- \alpha^*\theta)}} $ 
\begin{equation}\label{eq:thin} \liminf_{n \ra \infty} N_n^{-1/3} \log P \big( \OffAsymp \neq \emptyset \big)
\geq 
- \la \alpha^*  . \end{equation}
Now, by~\eqref{eq:b_large} we have that
\[ \sqrt{ \frac{\pi^2 \sigma^2}{2 b (1- \alpha^*\theta)}}  < \frac{ \log 2 }{2\sqrt{2} \, \alpha^*} , \]
so we can additionally assume that  $\la$ is small enough such that
\[ 2 \la \alpha^* < \log 2 . \]
Then, if we define
\[ r := \frac{1}{16} P ( \OffAsymp \neq \emptyset ) ,   \]
we obtain by~\eqref{eq:thin}  that 
\[ r^{-2} 2^{-N^{1/3}} \leq
e^{N^{1/3} ( 2 \la \alpha^* - \log 2 + o(1)) }  \ra 0 ,   \]

We will use the following general fact (see \cite[Fact 4.2]{GHS11}, but also \cite[Lemma 5.2]{HS09}): let $X_1, \ldots, X_k$ be independent non-negative random variables and suppose
$F: (0,\infty) \ra [0,\infty)$ is non-increasing, then 
\begin{equation}\label{eq:fact} E\Big[ \sum_{i=1}^k F(X_i) \, \Big|\, \sum_{i=1}^k X_i > 0 \Big] \leq \max_{1\leq i \leq k} E[ F(X_i) \, |\ X_i > 0 ] . \end{equation}

We will eventually apply Lemma~\ref{le:GHS-GW} to $\underline {\rm GW}_n$ and  thus  first estimate 
$P( 1 \leq |\underline\cC_n | \leq r^{-2} )$. We write $\underline S$ for the positions in the modified branching random walk used in the definition of $\underline {\rm GW}_n$.
For all sufficiently large $n$, by using~\eqref{eq:fact} in the second step, we obtain
\[ \begin{aligned} P( 1 \leq |\underline\cC_n |  &\leq r^{-2} )  \leq  P( 1 \leq |\underline\cC_n | \leq 2^{N^{1/3}} ) \\
& = P\Big(1 \leq  \sum_{x \in \OffAsymp} \#\{ y > x \, : \, 
|y| = L(N) \, , \,  \Delta \underline S(y_i) \leq M , i = N , \ldots, L(N) \} \leq 2^{N^{1/3}}\Big) \\
& \leq P  (\OffAsymp \neq \emptyset) P \big( \underline X_{\lfloor N^{1/3}\rfloor} \leq 2^{N^{1/3}} \, | \, \underline X_{\lfloor N^{1/3}\rfloor} \geq 1 \big)\\
& \leq P (\OffAsymp \neq \emptyset) \, 2^{-N^{1/3}}\\
& = P (\OffAsymp \neq \emptyset) \, r^2 o(1) , \end{aligned}
  \]
where we used Lemma~\ref{le:exp_growth} for the last inequality.

Finally, we can lower bound the survival probability of the killed IBRW by the survival probability of $\underline {\rm GW}_n$, where we recall that the distribution of $\underline {\rm GW}$ does not depend on the initial position and the initial type. Also, we deduce from the upper bound shown in Section~\ref{sec:upper_bound} that necessarily $P ( \underline {\rm GW}_n \neq \emptyset) \ra 0$, so that the assumption $r \leq \min\{ \frac{1}{8}, P (|\underline {\rm GW}_n|< \infty) \}$
is satisfied for large $n$ and we can apply Lemma~\ref{le:GHS-GW} to obtain
\[ \begin{aligned}
\zeta_n & \geq P( \underline {\rm GW}_n \neq \emptyset ) \\
& \geq P ( \underline \cC_n \neq \emptyset ) - 2 r^{-2} P ( 1 \leq |\underline \cC_n| \leq r^{-2}) - 2 r \\
& \geq (1-\underline q_m) P (\OffAsymp \neq \emptyset) - o(1) P (\OffAsymp \neq \emptyset) - 2 \frac{1}{16} P (\OffAsymp \neq \emptyset) \\
&\geq \frac 1 4 P (\OffAsymp \neq \emptyset) .
\end{aligned}
\]
Hence, we get from~\eqref{eq:thin} that
\[ \lim_{n \ra \infty} N_n^{-1/3} 
\log \zeta_n \geq - \la \alpha^*. 
\]
Finally, letting $\la \downarrow \sqrt{ \frac{\pi^2 \sigma^2}{2 b (1- \alpha\theta)}} $
and noting that $\eps_n = b N^{-2/3}(1+o(1))$,
we obtain 
\[ \lim_{n \ra \infty} \eps_n^{1/2}
\log \zeta_n \geq -  \alpha^*\sqrt{ \frac{\pi^2 \sigma^2}{2  (1- \alpha^*\theta)}}
\]
Hence,~\eqref{eq:compl} follows by letting $\theta \downarrow 0$.%
\end{proof}

\section{Proofs in the linear case}\label{sec:linear}

In this section, we show how to deduce  Corollary~\ref{linear_cor} from our general result, Theorem~\ref{thm:vary_f}.

\begin{proof}[Proof of Corollary~\ref{linear_cor}]
In the proof of Proposition 1.3 in \cite{DerMoe13} it was shown that for linear attachment functions, the spectral radius of $A_{\alpha}$ equals the largest eigenvalue of 
\[
\begin{pmatrix}
\frac{\beta}{\alpha-\gamma} & \frac{\beta}{1-\gamma-\alpha}\\
\frac{\beta+\gamma}{\alpha-\gamma} & \frac{\beta}{1-\gamma-\alpha}
\end{pmatrix}.
\]
This eigenvalue is given by
\begin{equation}\label{eq:spectral}
\rho(\alpha)=\frac{1}{2(1-\gamma-\alpha)(\alpha-\gamma)}\Big[\beta (1-2\gamma)+\sqrt{\beta^2 (1-2\gamma)^2+4\beta \gamma (1-\gamma-\alpha)(\alpha-\gamma)}\Big].
\end{equation}
In particular, $\alpha^*=\frac{1}{2}$ regardless of the choice of $\gamma \in [0,\frac{1}{2})$ and $\beta \in (0,1]$, and
\[
\rho(\alpha^*)=\frac{1}{\frac{1}{2}-\gamma}\Big[\beta+\sqrt{\beta^2+\beta \gamma}\Big].
\]
In order to apply Theorem~\ref{thm:vary_f}, we need to determine $\rho''(\alpha^*)$. To this end, we write $\Theta(\alpha)$ for the large squared bracket in \eqref{eq:spectral} and
\[
\varphi(\alpha)=(1-\gamma-\alpha)(\alpha-\gamma) \quad \Rightarrow \quad \varphi'(\alpha)=1-\gamma-\alpha-\alpha+\gamma =1-2 \alpha.
\]
Then
\begin{align}
\rho'(\alpha)&=-\frac{1}{2} \frac{\varphi'(\alpha)}{\varphi(\alpha)^2}( \beta(1-2\gamma)+ \Theta(\alpha)) +\frac{1}{2\varphi(\alpha)} \Big[0+\frac{1}{2} \big( \beta^2(1-2\gamma)^2+4 \beta \gamma \varphi(\alpha)\big)^{-\frac{1}{2}} \big(0+4 \beta \gamma (1-2\alpha)\big) \Big]\notag\\
&=\frac{2\alpha -1}{\varphi(\alpha)} \rho(\alpha) + \frac{4 \beta \gamma (1-2\alpha)}{4\varphi(\alpha)} \big( \beta^2(1-2\gamma)^2+4 \beta \gamma \varphi(\alpha)\big)^{-\frac{1}{2}}\notag\\
&=(2\alpha -1)\frac{1}{\varphi(\alpha)}\Big[ \rho(\alpha) - \frac{\beta \gamma}{ \sqrt{ \beta^2(1-2\gamma)^2+4 \beta \gamma \varphi(\alpha)}}\Big].\label{eq:1stder}
\end{align}
We need the second derivative only for $\alpha=\alpha^*=\frac{1}{2}$.
Since after applying the product rule, any term multiplied by $(2\alpha-1)$ vanishes, 
we obtain
\begin{equation}\label{2nd_der_rho}
\begin{split}
\rho''(\tfrac{1}{2})&= 
\frac{2}{\varphi(\alpha)}\Big[ \rho(\alpha) - \frac{\beta \gamma}{ \sqrt{ \beta^2(1-2\gamma)^2+4 \beta \gamma \varphi(\alpha)}}\Big]\bigg|_{\alpha=\frac{1}{2}}\\
&=\frac{2}{(\frac{1}{2}-\gamma)^2}\Big[ \rho(\alpha^*) - \frac{\beta \gamma}{ (1-2\gamma) \sqrt{\beta^2+ \beta \gamma}}\Big].
\end{split}
\end{equation}
The critical values $\betac(\gamma)$ and $\gammac(\beta)$ are chosen such that $\rho(\alpha^*)>1$ if and only if $\beta >\betac(\gamma)$ or $\gamma >\gammac(\beta)$. This implies
\begin{equation}\label{def_criticals}
\betac=\betac(\gamma)=\frac{(\frac{1}{2}-\gamma)^2}{1-\gamma}\quad \text{and} \quad \gamma_c=\gamma_c(\beta)=\frac{1}{2} \big(1-\beta -\sqrt{\beta^2+2\beta}\big).
\end{equation}
Clearly,
\begin{align}
&\sqrt{\betac^2+\betac\gamma}=\frac{1}{2}-\gamma-\betac=\frac{1}{2}\frac{\frac{1}{2}-\gamma}{1-\gamma},\label{formula_beta} \\  &\sqrt{\beta^2+\beta\gammac}=\tfrac{1}{2}-\gammac-\beta=\frac{1}{2}(\sqrt{\beta^2+2\beta}-\beta).\label{formula_gamma}
\end{align}
One easily checks that
\[
\frac{(\frac{1}{2}-\gammac(\beta))^2}{1-\gammac(\beta)}=\beta,
\]
i.e.\ $\betac(\gammac(\beta))=\beta$. We write $\rho(\alpha;\beta)$ or $\rho(\alpha;\gamma)$ to emphasize the dependence on $\beta$ or $\gamma$, respectively. By \eqref{2nd_der_rho}, \eqref{def_criticals} and \eqref{formula_beta},
\begin{align*}
\partial_{\alpha,\alpha}\rho(\alpha^*;\betac)&=\frac{2}{(\frac{1}{2}-\gamma)^2}\Big[ 1 - \frac{\betac \gamma}{ (1-2\gamma) \sqrt{\betac^2+ \betac \gamma}}\Big]=\frac{2}{(\frac{1}{2}-\gamma)^2}\Big[ 1 - \frac{(\frac{1}{2}-\gamma)^2}{1-\gamma} \frac{\gamma}{(\frac{1}{2}-\gamma) \frac{\frac{1}{2}-\gamma}{1-\gamma} }\Big]\\
&=\frac{2}{(\frac{1}{2}-\gamma)^2} [1 - \gamma]=2/\betac.
\end{align*}
By the continuity of $\rho$ and its derivatives in $\beta$ and $\gamma$, we obtain $\sigma^2=2/\betac(\gamma)$ for the convergence $\beta \downarrow \betac(\gamma)$ and $\sigma^2=2/\beta$ for $\gamma \downarrow \gammac(\beta)$. 
Moreover, two Taylor expansions yield for $\beta \downarrow \betac$
\begin{equation}\label{expan_eq}
\log \rho(\alpha^*;\beta)=\log(1+\partial_{\beta}\rho(\alpha^*;\beta_c) (\beta-\beta_c) +o(\beta-\beta_c)) = \partial_{\beta}\rho(\alpha^*;\beta_c) (\beta-\beta_c) (1+o(1)).\\
\end{equation}
The derivative is given by
\begin{align*}
\partial_{\beta}\rho(\alpha^*; \betac)&=\frac{1}{\frac{1}{2}-\gamma} \big(1 + \frac{ 2\betac +\gamma}{2\sqrt{\betac^2+\betac \gamma}}\big)= \frac{1}{\frac{1}{2}-\gamma} \Bigg(1 + \frac{ \frac{\frac{1}{2}-\gamma+\gamma^2}{1-\gamma}}{\frac{\frac{1}{2}-\gamma}{1-\gamma}}\Bigg)\\
&= \frac{1}{\frac{1}{2}-\gamma}  \frac{ 1-2\gamma+\gamma^2}{\frac{1}{2}-\gamma}=\left(\frac{1-\gamma}{\frac{1}{2}-\gamma}\right)^2.
\end{align*}
For the corresponding derivative with respect to $\gamma$ we use \eqref{formula_gamma} to derive
\begin{align*}
\partial_{\gamma}\rho(1/2; \gammac)&=\frac{1}{\frac{1}{2}-\gammac} \rho(1/2;\gamma_c)+\frac{1}{\frac{1}{2}-\gammac} \frac{\beta}{2\sqrt{\beta^2+\beta \gammac}}\\
&= \frac{1}{\frac{1}{2}-\gammac} (1+\frac{\beta}{\sqrt{\beta^2+2\beta}-\beta}) = \frac{2}{\sqrt{\beta^2+2\beta}+\beta} \frac{\sqrt{\beta^2+2\beta}-\beta+\beta}{\sqrt{\beta^2+2\beta}-\beta}\\
&= 2 \frac{\sqrt{\beta^2+2\beta}}{\beta^2+2\beta-\beta^2}= \frac{\sqrt{\beta^2+2\beta }}{\beta}.
\end{align*}
Repeating the argument of \eqref{expan_eq} for $\gamma$ instead of $\beta$, Theorem~\ref{thm:vary_f} yields
\begin{align*}
& \lim_{\beta \downarrow \betac(\gamma)} \sqrt{\beta-\betac} \log \size(\gamma \cdot +\beta)= -\sqrt{\frac{\pi^2}{2}} \sqrt{\frac{\sigma^2}{\partial_{\beta}\rho(\alpha^*;\betac)}} \alpha^*=-\frac{\pi}{2}\sqrt{\frac{2\frac{2}{\betac}}{\frac{1-\gamma}{\betac}}} \frac{1}{2}= - \frac{\pi}{2\sqrt{1-\gamma}},\\
& \lim_{\gamma \downarrow \gammac(\beta)} \sqrt{\gamma-\gammac} \log \size( \gamma\cdot +\beta)= -\sqrt{\frac{\pi^2}{2}} \sqrt{\frac{\sigma^2}{\partial_{\gamma}\rho(\alpha^*;\gammac)}} \alpha^*=-\frac{\pi}{2} \sqrt{\frac{2\frac{2}{\beta}}{\frac{\sqrt{\beta^2+2\beta}}{\beta}}} \frac{1}{2}=-\frac{\pi}{2(\beta^2+2\beta)^{\frac{1}{4}}}.\qedhere
\end{align*}
\end{proof}

\appendix

\section{Appendix: Proof of Mogulskii's Theorem}\label{mogulskii_proof}

This section is devoted to the proof of Theorem~\ref{thm:mogulskii}.

\begin{proof}
The proof is an adaptation of the proof presented in \cite{GHS11}. We give a detailed proof of the upper bound. The changes for the lower bound are similar. 

Let $N=\lfloor \frac{k_n}{r_n}\rfloor$ and $m_N=k_n$, $m_k=k r_n$ for $0 \le k \le N-1$. The Markov property implies that
\begin{align*}
&\P_{(0,\type_0)}(E_n)=\P_{(0,\type_0)}\Big(\bigcap_{k=1}^N \{ g_1\big(\frac{i}{k_n}\big) \le \frac{S_i^{(n)}}{a_n} \le g_2\big(\frac{i}{k_n}\big) \, \forall\, i \in (m_{k-1},m_k]\cap \N \}\Big)\\
&\le \prod_{k=2}^{N-1} \sup_{x \in [g_1(\frac{m_{k-1}}{k_n}),g_2(\frac{m_{k-1}}{k_n})]} \sup_{\type \in \typespace} \P_{(0,\type)}\Big(g_1(\tfrac{m_{k-1}+i}{k_n}) \le \frac{S_i^{(n)}}{a_n}+x\le g_2(\tfrac{m_{k-1}+i}{k_n})\, \forall \,i\in [r_n]\Big).
\end{align*}
Since $g_1$ and $g_2$ are bounded, for every $\delta >0$ there exists $K \in \N$ such that $[g_1(\frac{m_{k-1}}{k_n}),g_2(\frac{m_{k-1}}{k_n})] \subseteq [-K\delta, (K-1)\delta] =\bigcup_{j=-K}^{K-1} [j\delta,(j+1)\delta]$. Continuity of $g_1$ and $g_2$ further guarantees the existence of $A=A(\delta)>0$ such that
\begin{equation}\label{def_A_eq}
\sup_{0 \le s,t\le 1\colon |s-t|\le \frac{2}{A}} |g_1(t)-g_1(s)|+|g_2(t)-g_2(s)|< \delta.
\end{equation}
Let $J_{l}=(\frac{(l-1)(N-2)}{A}+1,\frac{l(N-2)}{A}+1] \cap \N$ for $j\in [A]$. We show that for sufficiently large $n$, for all $i \in [r_n]$ and $l \in [A]$
\[
\Big|\frac{i+m_{k-1}}{k_n}-\frac{l}{A}\Big|\le \frac{2}{A} \quad \text{for all } k \in \Big\{\frac{(l-1)(N-1)}{A}+1,\ldots,\frac{l(N-2)}{A}+1\Big\}.
\]
Since $\frac{k_n}{r_n} -  (N-1) \in [1,2]$, we have
\begin{align*}
&\frac{i+m_{k-1}}{k_n} \ge \frac{1+r_n \frac{(l-1)(N-1)}{A}}{k_n} \ge \frac{l-2}{A} \qquad \text{for all } l\in [A]\\
\Leftrightarrow \; &  A \ge k_n(l-2)-r_n(l-1)(N-1)=-k_n +(l-1)r_n \Big(\frac{k_n}{r_n}-(N-1)\Big) \quad \text{for all }l\in [A]\\
\Leftarrow\; & A\ge -k_n +(A-1)r_n 2.
\end{align*}
Since $\lim_{n \to \infty}r_n/k_n = 0$, this is satisfied for large $n$. For the other direction, we use that $\frac{r_n}{k_n} \le \frac{1}{N}$ to see that
\begin{align*}
&\frac{i+m_{k-1}}{k_n} \le \frac{kr_n}{k_n} \le \Big(\frac{l(N-2)}{A}+1\Big)\frac{r_n}{k_n} \le \frac{l+2}{A} \qquad\text{for all } l\in [A]\\
\Leftrightarrow \;&  A\frac{r_n}{k_n} \le l+2-l(N-2)\frac{r_n}{k_n}=l\Big(1-(N-2)\frac{r_n}{k_n}\Big)+2 \quad \text{for all } l\in [A]\\
\Leftarrow \;&  A\frac{r_n}{k_n} \le 2.
\end{align*}
The last inequality holds since $\frac{r_n}{k_n} \to 0$. The small $k$-values are needed later for the proof of the lower bound. For $k \in J_l$, \eqref{def_A_eq} implies that for all $x \in [j \delta, (j+1)\delta]$
\begin{align*}
&\sup_{\type \in \typespace} \P_{(0,\type)} \bigg(g_1\Big(\frac{m_{k-1}+i}{k_n}\Big)\le \frac{S_i^{(n)}}{a_n}+x\le g_2\Big(\frac{m_{k-1}+i}{k_n}\Big) \, \forall i\in [r_n]\bigg)\\
&\le \sup_{\type \in \typespace} \P_{(0,\type)} \Big(g_1\Big(\frac{l}{A}\Big)-(j+2)\delta \le \frac{S_i^{(n)}}{a_n}\le g_2\Big(\frac{l}{A}\Big)-(j-1)\delta \, \forall i\in [r_n]\Big)=:q_{l,n}(j).
\end{align*}
Since $\bigcup_{l=1}^A J_l =(1,N-1]\cap \N$, we derive
\[
\P_{(0,\type_0)}(E_n)\le \prod_{l=1}^A \Big[\max_{j \in [-K,K) \cap \Z} q_{l,n}(j)\Big]^{|J_l|}.
\]
The assumed uniform convergence implies that for $n \to \infty$
\[
q_{l,n}(j) \to P\Big(g_1\Big(\frac{l}{A}\Big)-(j+2)\delta \le \sqrt{\sigma^2A}W_t \le g_2\Big(\frac{l}{A}\Big)-(j-1)\delta\, \forall t \in [0,1]\Big).
\]
The right-hand side can be estimated by (see for example \cite{IMK74} p.31 or \cite{GHS11} Eq.(5.4))
\[
\exp\Big(-\frac{\pi^2 \sigma^2}{2} \frac{(1-\delta)A}{[g_2(\frac{l}{A})-g_1(\frac{l}{A})+3\delta]^2}\Big).
\]
Since there are only finitely many $(j,l) \in [-K,K) \cap \N\times \{1,\ldots,A\}$, the convergence and the bound hold uniformly in $(j,l)$. Moreover,
\[
\# J_l \ge \frac{l(N-2)}{A}-\frac{(l-1)(N-2)}{A} -1 =\frac{N-2}{A} -1\ge \frac{\frac{k_n}{r_n}-1-2}{A}-1\ge \frac{k_n}{a_n^2 A^2}-\frac{3}{A}-1.
\]
Now we collected everything needed to derive
\begin{align*}
\limsup_{n \to \infty}\frac{a_n^2}{k_n} \log P_{(0,\type_0)}(E_n) & \le \limsup_{n \to \infty}\frac{a_n^2}{k_n} \sum_{l=1}^A \# J_l \log \max_{j \in [-K,K) \cap \Z} \tilde{q}_{l,n}(j)\\
&\le \sum_{l=1}^A \limsup_{n \to \infty} \frac{a_n^2}{k_n} \Big(\frac{n}{a_n^2 A^2}-\frac{3}{A}-1\Big) \log \max_{j \in [-K,K) \cap \Z} \tilde{q}_{l,n}(j)\\
&\le \frac{1}{A^2}\sum_{l=1}^A -\frac{\pi^2 \sigma^2}{2} \frac{(1-\delta)A}{[g_2(\frac{l}{A})-g_1(\frac{l}{A})+3\delta]^2}.
\end{align*} 
Since $g_1$ and $g_2$ are continuous functions, taking $A\to \infty$ yields
\[
\limsup_{n \to \infty}\frac{a_n^2}{k_n} \log P_{(0,\type_0)}(E_n)\le  -\frac{\pi^2 \sigma^2}{2}  \int_0^1 \frac{(1-\delta)}{[g_2(x)-g_1(x)+3\delta]^2}\, dx.
\]
Now we can take $\delta \to 0$ to establish the claim.

\emph{Sketch of the lower bound:} Choose a continuous function $g\colon[0,1] \to \R$ such that $g_1(t) < g(t) <g_2(t)$ for all $t \in [0,1]$. Since it suffices to prove the lower bound for $b$ small, we can assume that $g(1) \geq g_2(1) -b$.
Then, let $\delta >0$ be such that $g(t)-g_1(t)>3\delta$ and $g_2(t)-g(t) >9\delta$ for all $t\in [0,1]$. Moreover, $A$ is chosen large enough that
\[
\sup_{0 \le s \le t \le 1\colon t-s \le \frac{2}{A}} \Big(|g_1(t)-g_1(s)|+|g(t)-g(s)|+|g_2(t)-g_2(s)|\Big) \le \delta.
\]
Choose $N=\lfloor \frac{k_n}{r_n}\rfloor$, $m_N=k_n$ and $m_k=k r_n$ for $0 \le k \le N-1$. Writing $y_k=g(\frac{m_k}{k_n})$ for $1\le k \le N$, the Markov property implies
\[
P_{0,\type_0}(E_n) \ge p_{1,n}(0,\type_0) \times \prod_{k=2}^{N} \inf_{y \in [y_{k-1},y_{k-1}+6 \delta]} \inf_{\type \in \typespace} p_{k,n}(y,\type),
\]
where for $1 \le k \le N$, $y \in \R$, and $\type \in \typespace$
\begin{align*}
p_{k,n}(y,\type)&=P_{(0,\type)}\Big(\alpha_{i,k,n} \le \frac{S_i^{(n)}}{a_n} +y \le \beta_{i,k,n}, \forall i \in [m_k-m_{k-1}];  y_k \leq S^\ssup{n}_{m_k-m_{k-1}} + 6 \delta \Big)\\
&\alpha_{i,k,n}:= g_1\Big(\frac{i+m_{k-1}}{k_n}\Big) \quad \beta_{i,k,n}:=g_2\Big(\frac{i+m_{k-1}}{k_n}\Big).
\end{align*}
Now the choice of parameters implies that for $n$ sufficiently large
\[
P_{0,\type_0}(E_n) \ge \min\{p_{N,n}^1,p_{N,n}^2\} \prod_{l =1}^{A} \Big(\min\{q_{l,n}^1,q_{l,n}^2\}\Big)^{k_n/[(Aa_n^2-1)A]},
\]
where $p_{N,n}^1$ is the infimum over all $\type \in \typespace$ of $p_{N,n}(y_{k-1},\type)$ where $\beta_{i,N,n}$ is replaced by $\beta_{i,N,n}-3 \delta$ and the same for $p_{N,n}^2$ with $\alpha_{i,N,n}$ replaced by $\alpha_{i,N,n}-3 \delta$ and $\beta_{i,N,n}$ by $\beta_{i,N,n}-6\delta$. Moreover,
\[\begin{aligned}
q_{l,n}^{(1)}& =\inf_{\type \in \typespace}P_{(0,\type)}\Big(g_1\Big(\frac{l}{A}\Big)-g\Big(\frac{l}{A}\Big)+2 \delta \le \frac{S_i^{(n)}}{a_n} \le g_2\Big(\frac{l}{A}\Big)-g\Big(\frac{l}{A}\Big)-5\delta\; \forall \,i \in [r_n], \delta \leq \frac{S_{r_n}^\ssup{n}}{a_n} \leq 2 \delta \Big),\\
q_{l,n}^{(2)}& =\inf_{\type \in \typespace}P_{(0,\type)}\Big(g_1\Big(\frac{l}{A}\Big)-g\Big(\frac{l}{A}\Big)- \delta \le \frac{S_i^{(n)}}{a_n} \le g_2\Big(\frac{l}{A}\Big)-g\Big(\frac{l}{A}\Big)-8\delta\; \forall \,i \in [r_n], -2\delta \leq \frac{S_{r_n}^\ssup{n}}{a_n} \leq - \delta \Big)
\end{aligned} 
\]
The claim now follows with the same arguments as in \cite{GHS11} because of the assumed convergence of the probabilities.
\end{proof}

\bibliographystyle{abbrv}

\end{document}